\documentclass[11pt,reqno]{amsart}
\usepackage{amsmath,amssymb,mathrsfs}
\usepackage{graphicx,cite}
\usepackage{subfigure}
\usepackage{epstopdf}
\usepackage{enumerate}
\usepackage{graphics} 
\usepackage{epsfig} 
\usepackage{tikz}
\usepackage{paralist}
\usepackage{booktabs}
\usepackage{stmaryrd}
\usepackage{appendix}

\newtheorem{theorem}{Theorem}[section]

\newtheorem{problem}{Problem}

\numberwithin{equation}{section}

\begin{document}
	\title{Inverse obstacle scattering for elastic waves in the time domain}	
	
	\author{Lu Zhao}
	\email{zhaolu18@mails.jlu.edu.cn}
	
	\author{Heping Dong}
	\email{dhp@jlu.edu.cn}
	
	\author{Fuming Ma}
	\email{mfm@jlu.edu.cn}
	
	\thanks{The work of second author was partially supported by the National Natural Science Foundation of China (grants 11801213) and the National Key Research and Development Program of China (grants 2020YFA0713602). The work of third author was partially supported by the National Natural Science Foundation of China (grants 11771180).}

	\begin{abstract}
		This paper concerns an inverse elastic scattering problem which is to determine a rigid obstacle from time domain scattered field data for a single incident plane wave. By using Helmholtz decomposition, we reduce the initial-boundary value problem of the time domain Navier equation to a coupled initial-boundary value problem of wave equations, and prove the uniqueness of the solution for the coupled problem by employing energy method. The retarded single layer potential is introduced to establish the coupled boundary integral equations, and the uniqueness is discussed for the solution of the coupled boundary integral equations. Based on the convolution quadrature method for time discretization, the coupled boundary integral equations are reformulated into a system of boundary integral equations in $s$-domain, and then a convolution quadrature based nonlinear integral equation method is proposed for the inverse problem. Numerical experiments are presented to show the feasibility and effectiveness of the proposed method.
	\end{abstract}
\maketitle
	\section{Introduction}
	The time domain inverse obstacle scattering problem is an important research topic with many potential applications such as sonar detection, geophysical exploration, biomedical imaging and noninvasive detecting. Since the time domain broadband signals can be captured easier and these measurement data contain more information at discrete frequencies compared with frequency domain data, recently, time domain scattering and inverse scattering problems have attracted a lot of attention \cite{ChenBo, Hsiao, Hu, Ji, Liu, LiuHu, Sini, Wang}.

	The time domain inversion algorithms consist of qualitative and quantitative method. The advantage of the qualitative method lies in the fact that it avoids the need of priori information of the obstacle and the solution of a sequence of direct problems. Many numerical methods fall into this category, such as the point source method \cite{Luke}, the probe method \cite{Burkard}, the factorization method \cite{Cakoni}, the enclosure method \cite{Ikehata1}, the strengthened total focusing method \cite{Guo} and the linear sampling method \cite{Chen1,Chen2,Guo1,Guo2,Hadder}. In order to obtain finer reconstruction, a quantitative method, namely convolution quadrature based nonlinear integral equation method, is proposed for the inverse acoustic obstacle scattering problem in \cite{Zhao}. The convolution quadrature method for time discretization was proposed by Lubich \cite{Lubich19881,Lubich19882,Lubich1992,Lubich1994}, and was extended to solve the acoustic scattered field by combining various numerical methods on space discretization \cite{CQ,Dominguez,sayas}. It provides a straightforward way to deal with time variables by using the Laplace transform of the kernel function. The nonlinear integral equation method proposed by Johansson and Sleeman \cite{Johansson} belongs to a simplified Newton method, which also has been applied in phaseless inverse problems \cite{lambda,Dong2019,Dong2020}.
	
	The methods mentioned above are commonly employed for time domain inverse acoustic obstacle scattering problem. In this work, we mainly concern the inverse elastic scattering problem of reconstructing a rigid obstacle from time domain scattered field data for a single incident plane wave. The obstacle is assumed to be embedded in a homogeneous and isotropic medium, and the scattered field data is measured on a circle which has a finite distance away from the obstacle. Motivated by the recent works \cite{Bao,Zhao}, we propose a convolution quadrature based nonlinear integral equation method to solve this inverse problem.
	Specifically, by using Helmholtz decomposition, we convert the model problem into a coupled initial-boundary value problem of wave equations, and we prove the uniqueness of the solution for this coupled problem by employing energy method. Then, we establish coupled boundary integral equations with the help of retarded layer potential and prove the uniqueness of the solution for the coupled boundary integral equations.	
	Based on the convolution quadrature method \cite{CQ, Lubich1994} for time discretization, we reformulate the coupled boundary integral equations into a system of boundary integral equations in $s$-domain and make use of Nystr$\mathrm{\ddot o}$m-type discretization \cite{Dong2019} for the boundary integral equations in $s$-domain. Finally, we combine the nonlinear integral equation method with convolution quadrature technique to reconstruct the obstacle by using the time domain scattered field data for a single incident plane wave. To our best knowledge, this is the first work on inverse obstacle scattering problem for elastic waves in the time domain. The goal of this work is threefold:
	\begin{itemize}
		\item [(1)] prove the uniqueness of the solution for the coupled initial-boundary value problem by employing energy method;
		\item [(2)] establish coupled boundary integral equations by using retarded layer potential and prove the uniqueness of the solution for the coupled boundary integral equations;
		\item [(3)] propose a convolution quadrature based nonlinear integral equation method to reconstruct the obstacle from the time domain scattered field data for a single incident plane wave.
	\end{itemize}
	
	The paper is organized as follows. In section 2, we convert the model problem into a coupled initial-boundary value problem of wave equations and prove the uniqueness for the coupled problem. In section 3, we deduce the coupled boundary integral equations and establish the uniqueness of the solution for the coupled boundary integral equations. Then the convolution quadrature method is applied to obtain a system of boundary integral equations in $s$-domain. Section 4
	presents the Nystr$\mathrm{\ddot o}$m-type discretization of the 
	boundary integral equations in $s$-domain. In section 5, the convolution quadrature based nonlinear integral equation method is proposed to solve the inverse obstacle scattering problem in the time domain. The numerical experiments are shown in section 6 to validate the feasibility of the proposed method. The paper is concluded in section 7.
	\section{Problem formulation}
	Consider a two-dimensional rigid obstacle, which is described as a bounded domain $D \subset \mathbb{R}^2$ with smooth boundary $\partial D$. Denote by $\nu=(\nu_1,\nu_2)^\top$ and $\tau=(\tau_1,\tau_2)^\top$ the unit normal and tangential vectors on $\partial D$, respectively, where $\tau_1=-\nu_2$, $\tau_2=\nu_1$. The exterior domain $\mathbb{R}^2\setminus \overline{D}$ is assumed to be filled with homogeneous and isotropic elastic medium with a unit mass density. Picking an appropriate constant $R>0$, we define $B_R=\left\lbrace \boldsymbol{x}\in\mathbb{R}^2:|\boldsymbol{x}|<R\right\rbrace $ such that $\overline{D}\subset B_R$, and let the boundary $\partial B_R$ be the observation curve.
	
	Let the obstacle be illuminated by a time domain plane wave $\boldsymbol{u}^{\rm inc}$, which satisfies the time domain two-dimensional Navier equation
	\[
	\partial_t^2\boldsymbol{u}^{\rm inc}-\mu\Delta\boldsymbol{u}^{\rm inc}-(\lambda+\mu)\nabla\nabla\cdot\boldsymbol{u}^{\rm inc}=\boldsymbol{0},\quad {\rm in}~\mathbb{R}^2\times(0,\infty),
	\]
	where $\lambda$ and $\mu$ are the Lam$\acute{\rm e}$ constants satisfying $\mu>0$ and $\lambda+\mu>0$. The incident wave can be either a compressional plane wave:
	\[
	\boldsymbol{u}^{\rm inc}(\boldsymbol{x},t)=\boldsymbol{d}f\left( \boldsymbol{x}\cdot \boldsymbol{d}+c_1t-T_0\right)
	\]
	or a shear plane wave:
	\[
	\boldsymbol{u}^{\rm inc}(\boldsymbol{x},t)=\boldsymbol{d}^\perp f\left( \boldsymbol{x}\cdot \boldsymbol{d}+c_2t-T_0\right),
	\]
	where $\boldsymbol{d}=\left( \cos \theta,\sin \theta\right) ^\top$, $\boldsymbol{d}^\perp=\left( -\sin \theta,\cos \theta\right) ^\top$, $\theta \in \left[ 0,2\pi\right) $ is the incident angle, and we denote the wave speed by $c_1=\left( \lambda+2\mu\right) ^{1/2}$ and $c_2=\mu^{1/2}$. In addition, we assume that $f$ is a causal function and $T_0$ is chosen such that the supports of $\boldsymbol{u}^{\rm inc}(\cdot,0)$ and $\partial_t\boldsymbol{u}^{\rm inc}(\cdot,0)$ do not intersect with $\overline{D}$ at initial time.
	
	The displacement of the total wave field $\boldsymbol{u}$ also satisfies the time domain Navier equation:
	\[
	\partial_t^2\boldsymbol{u}-\mu\Delta\boldsymbol{u}-(\lambda+\mu)\nabla\nabla\cdot\boldsymbol{u}=\boldsymbol{0},\quad {\rm in}~\mathbb{R}^2\setminus\overline{D}\times(0,\infty).
	\]
	Since the obstacle is assumed to be rigid, it holds that
	\[
	\boldsymbol{u}=\boldsymbol{0}, \quad{\rm on}~\partial D\times(0,\infty).
	\]
	In addition, $\boldsymbol{u}$ satisfies the initial condition
	\begin{equation*}
		\left\{
		\begin{aligned}
			&\boldsymbol{u}|_{t=0}=\boldsymbol{u}^{\rm inc}|_{t=0}=\boldsymbol{u}_0&&{\rm in} R^2\setminus\overline{D},\\
			&\partial_t\boldsymbol{u}|_{t=0}=\partial_t\boldsymbol{u}^{\rm inc}|_{t=0}=\boldsymbol{v}_0&&{\rm in} R^2\setminus\overline{D}.
		\end{aligned}
		\right.
	\end{equation*}
	The total field $\boldsymbol{u}$ consists of the incident field $\boldsymbol{u}^{\rm inc}$ and the scattered field $\boldsymbol{v}$, i.e.
	\[
	\boldsymbol{u}=\boldsymbol{u}^{\rm inc}+\boldsymbol{v}.
	\]
	It is easy to verify that the scattered field $\boldsymbol{v}$ satisfies the initial-boundary value problem
	\begin{equation}\label{scattered field}
		\left\{
		\begin{aligned}
			&\partial_{t}^2\boldsymbol{v}-\mu\Delta \boldsymbol{v}-(\lambda+\mu)\nabla\nabla\cdot\boldsymbol{v}=\boldsymbol{0}&& {\rm in}~\mathbb{R}^2\setminus\overline{D}\times(0,\infty),\\
			&\boldsymbol{v}=-\boldsymbol{u}^{\rm inc}&&{\rm on}~\partial D\times(0,\infty),\\
			&\boldsymbol{v}(\cdot,0)=\partial_t \boldsymbol{v}(\cdot,0)=\boldsymbol{0}&&{\rm in}~ \mathbb{R}^2\setminus\overline{D}.
		\end{aligned}
		\right.
	\end{equation}
	
	Given a vector function $\boldsymbol{w}=(w_1,w_2)^\top$ and a scalar function $w$, we define the scalar and vector curl operators:
	\[
	{\rm curl}\boldsymbol{w}=\partial_{x_1}w_2-\partial_{x_2}w_1,\quad {\rm \mathbf{curl}}w=(\partial_{x_2}w,-\partial_{x_1}w)^\top.
	\]
	For any solution $\boldsymbol{v}$ of the elastic wave equation \eqref{scattered field}, the Helmholtz decomposition \cite{Bao} reads
	\begin{equation}\label{Helmholtz decomposition}
		\boldsymbol{v}=\nabla \phi+{\rm \mathbf{curl}}\psi,
	\end{equation}
	where $\phi$, $\psi$ are scalar potential functions. Combining \eqref{Helmholtz decomposition} and \eqref{scattered field}, we may obtain
	\[
	\nabla(\partial_{t}^2\phi-(\lambda+2\mu)\Delta\phi)+{\rm \mathbf{curl}}(\partial_{t}^2\psi-\mu\Delta\psi)=\boldsymbol{0},
	\]
	that means $\phi$ and $\psi$ satisfy the wave equations
	\[
	\Delta\phi-\frac{1}{c_1^2}\partial_{t}^2\phi=0,\quad \Delta\psi-\frac{1}{c_2^2}\partial_t^2\psi=0
	\]
	and the initial conditions
	\[
	\phi|_{t=0}=\partial_t\phi|_{t=0}=0,\quad \psi|_{t=0}=\partial_t\psi|_{t=0}=0,
	\]
	where $c_1=(\lambda+2\mu)^{1/2}$, $c_2=\mu^{1/2}$. By using the Helmholtz decomposition and boundary condition on $\partial D$, we have
	\[
	\boldsymbol{v}=\nabla \phi+{\rm \mathbf{curl}}\psi=-\boldsymbol{u}^{\rm inc}.
	\]
	Taking the dot product of the above equation with $\nu$ and $\tau$, respectively, we get
	\[
	\partial_{\nu}\phi+\partial_{\tau}\psi=f_1,\quad \partial_{\tau}\phi-\partial_{\nu}\psi=f_2,
	\]
	where 
	\[
	f_1=-\nu\cdot \boldsymbol{u}^{\rm inc},\quad f_2=-\tau\cdot \boldsymbol{u}^{\rm inc}.
	\]
	In summary, the scalar potential functions $\phi$, $\psi$ satisfy the coupled initial-boundary value problem
	\begin{equation*}
		\left\{
		\begin{aligned}
			&\Delta \phi-\frac{1}{c_1^2}\partial_{t}^2\phi=0,\quad \Delta \psi-\frac{1}{c_2^2}\partial_{t}^2\psi=0 &&{\rm in}~\mathbb{R}^2\setminus\overline{D}\times(0,\infty),\\
			&\partial_{\nu}\phi+\partial_{\tau}\psi=f_1,\quad \partial_{\tau}\phi-\partial_{\nu}\psi=f_2 &&{\rm on}~\partial D\times(0,\infty),\\
			&\phi|_{t=0}=\partial_t\phi|_{t=0}=0,\quad\psi|_{t=0}=\partial_t\psi|_{t=0}=0 &&{\rm in}~\mathbb{R}^2\setminus\overline{D}.
		\end{aligned}
		\right.
	\end{equation*}
	
Since the wave has finite speed of propagation, for any given time $T>0$, we can pick a sufficiently large $\tilde{R}>0$ such that the scattered field do not reach the curve $\partial B_{\tilde{R}}=\left\lbrace \boldsymbol{x}\in\mathbb{R}^2: |\boldsymbol{x}|=\tilde{R}\right\rbrace $ at time $T$, that means 
\[
\phi=\psi=0,\quad {\rm in}~ \mathbb{R}^2\setminus B_{\tilde{R}}\times(0,T].
\]
Thus we have
	\begin{equation}\label{initial-boundary value problem}
		\left\{
		\begin{aligned}
			&\Delta \phi-\frac{1}{c_1^2}\partial_{t}^2\phi=0,\quad \Delta \psi-\frac{1}{c_2^2}\partial_{t}^2\psi=0 &&{\rm in}~\Omega\times(0,T],\\
			&\partial_{\nu}\phi+\partial_{\tau}\psi=f_1,\quad \partial_{\tau}\phi-\partial_{\nu}\psi=f_2 &&{\rm on}~\partial D\times(0,T),\\
			&\phi=0,\quad\psi=0&&{\rm on}~\partial B_{\tilde{R}}\times(0,T),\\
			&\phi|_{t=0}=\partial_t\phi|_{t=0}=0,\quad\psi|_{t=0}=\partial_t\psi|_{t=0}=0 &&{\rm in}~\Omega,
		\end{aligned}
		\right.
	\end{equation}
	where $\Omega=B_{\tilde{R}}\setminus \overline{D}$.
	\begin{theorem}\label{problemuniqueness}
	The coupled initial-boundary value problem \eqref{initial-boundary value problem} has at most one solution for $\phi\in L^2(0,T;H^2(\Omega))$, $\psi\in L^2(0,T;H^2(\Omega))$.
\end{theorem}
\begin{proof}
	It suffices to show that $\phi=\psi=0$ in $\Omega\times (0,T]$ if $f_1=f_2=0$ on $\partial D\times(0,T]$. For $\phi\in C^{\infty}(0,T;H^2(\Omega))$ and $\psi\in C^{\infty}(0,T;H^2(\Omega))$, according to the wave equations in \eqref{initial-boundary value problem} and  $c_1=(\lambda+2\mu)^{1/2}$, $c_2=\mu^{1/2}$, we obtain
		\begin{equation*}
			\begin{aligned}
				0=&\int_{0}^{t}\int_{\Omega}\left[ \nabla\left( \partial_{\tilde{t}}^2\phi-\left( \lambda+2\mu\right) \Delta\phi\right)\right]  \cdot\left[ \partial_{\tilde{t}}\left( \nabla\phi+{\rm \mathbf{curl}}\psi\right)\right]  \mathrm{d}\boldsymbol{x}\mathrm{d}\tilde{t}\\
				&+\int_{0}^{t}\int_{\Omega}\left[ {\rm \mathbf{curl}}\left( \partial_{\tilde{t}}^2\psi-\mu \Delta\psi\right)\right]  \cdot\left[ \partial_{\tilde{t}}\left( \nabla\phi+{\rm \mathbf{curl}}\psi\right)\right]  \mathrm{d}\boldsymbol{x}\mathrm{d}\tilde{t}\\
				=&A_1+A_2+A_3.
			\end{aligned}
		\end{equation*}		
where we have set
\begin{equation*}
   \begin{aligned}
A_1:=&\int_{0}^{t}\int_{\Omega}\partial_{\tilde{t}}^2\nabla\phi\cdot\partial_{\tilde{t}}\nabla\phi\mathrm{d}\boldsymbol{x}\mathrm{d}\tilde{t}+\int_{0}^{t}\int_{\Omega}\partial_{\tilde{t}}^2\nabla\phi\cdot\partial_{\tilde{t}}{\rm \mathbf{curl}}\psi\mathrm{d}\boldsymbol{x}\mathrm{d}\tilde{t}\\
+&\int_{0}^{t}\int_{\Omega}\partial_{\tilde{t}}^2{\rm \mathbf{curl}}\psi\cdot\partial_{\tilde{t}}{\rm\mathbf{curl}}\psi\mathrm{d}\boldsymbol{x}\mathrm{d}\tilde{t}+\int_{0}^{t}\int_{\Omega}\partial_{\tilde{t}}^2{\rm \mathbf{curl}}\psi\cdot\partial_{\tilde{t}}\nabla\phi\mathrm{d}\boldsymbol{x}\mathrm{d}\tilde{t},\\
A_2:=&-(\lambda+\mu)\int_{0}^{t}\int_{\Omega}(\nabla\Delta\phi)\cdot\partial_{\tilde{t}}\nabla\phi\mathrm{d}\boldsymbol{x}\mathrm{d}\tilde{t}-(\lambda+\mu)\int_{0}^{t}\int_{\Omega}(\nabla\Delta\phi)\cdot\partial_{\tilde{t}}{\rm\mathbf{curl}}\psi\mathrm{d}\boldsymbol{x}\mathrm{d}\tilde{t},\\	
A_3:=&-\mu\int_{0}^{t}\int_{\Omega}(\nabla\Delta\phi)\cdot\partial_{\tilde{t}}\nabla\phi\mathrm{d}\boldsymbol{x}\mathrm{d}\tilde{t}-\mu\int_{0}^{t}\int_{\Omega}(\nabla\Delta\phi)\cdot\partial_{\tilde{t}}{\rm\mathbf{curl}}\psi\mathrm{d}\boldsymbol{x}\mathrm{d}\tilde{t}\\
&-\mu\int_{0}^{t}\int_{\Omega}({\rm \mathbf{curl}}\Delta\psi)\cdot\partial_{\tilde{t}}\nabla\phi\mathrm{d}\boldsymbol{x}\mathrm{d}\tilde{t}-\mu\int_{0}^{t}\int_{\Omega}({\rm \mathbf{curl}}\Delta\psi)\cdot\partial_{\tilde{t}}{\rm\mathbf{curl}}\psi\mathrm{d}\boldsymbol{x}\mathrm{d}\tilde{t}.
\end{aligned}
\end{equation*}
Using the Green's theorem and the initial condition in \eqref{initial-boundary value problem}, we get
		\begin{equation*}
			\begin{aligned}
				A_1=&\int_{\Omega}\int_{0}^{t}\frac{1}{2}\frac{\mathrm{d}}{\mathrm{d}\tilde{t}}\left( \partial_{\tilde{t}}\nabla\phi\cdot\partial_{\tilde{t}}\nabla\phi\right) +\frac{\mathrm{d}}{\mathrm{d}\tilde{t}}\left(\partial_{\tilde{t}}\nabla\phi\cdot\partial_{\tilde{t}}{\rm\mathbf{curl}}\psi\right) +\frac{1}{2}\frac{\mathrm{d}}{\mathrm{d}\tilde{t}}\left(\partial_{\tilde{t}}{\rm\mathbf{curl}}\psi\cdot\partial_{\tilde{t}}{\rm \mathbf{curl}}\psi\right) \mathrm{d}\tilde{t}\mathrm{d}\boldsymbol{x}\\
				=&\int_{\Omega}\frac{1}{2}\int_{0}^{t}\frac{\mathrm{d}}{\mathrm{d}\tilde{t}}\left[ \left( \partial_{\tilde{t}}\nabla\phi+\partial_{\tilde{t}}{\rm\mathbf{curl}}\psi\right) \cdot\left( \partial_{\tilde{t}}\nabla\phi+\partial_{\tilde{t}}{\rm\mathbf{curl}}\psi\right)\right] \mathrm{d}\tilde{t}\mathrm{d}\boldsymbol{x}\\
				=&\frac{1}{2}\int_{\Omega} \left( \partial_{t}\nabla\phi+\partial_{t}{\rm\mathbf{curl}}\psi\right)\cdot\left( \partial_{t}\nabla\phi+\partial_{t}{\rm\mathbf{curl}}\psi\right)\mathrm{d}\boldsymbol{x}\\
				=&\frac{1}{2}\left\| \partial_{t}\nabla\phi+\partial_{t}{\rm\mathbf{curl}}\psi\right\| _{L^2(\Omega)^2}^2,
			\end{aligned}
		\end{equation*}
		\begin{equation*}
			\begin{aligned}
				A_2=&(\lambda+\mu)\int_{0}^{t}\int_{\Omega}\Delta \phi\partial_{\tilde{t}}\Delta\phi\mathrm{d}\boldsymbol{x}\mathrm{d}\tilde{t}\\
				&-(\lambda+\mu)\int_{0}^{t}\int_{\partial B_{\tilde{R}}}\left[ \Delta\phi\nu_{\partial B_{\tilde{R}}}\right] \cdot\partial_{\tilde{t}}\nabla\phi\mathrm{d}s\mathrm{d}\tilde{t}+(\lambda+\mu)\int_{0}^{t}\int_{\partial D}\left[ \Delta\phi\nu_{\partial D}\right] \cdot\partial_{\tilde{t}}\nabla\phi\mathrm{d}s\mathrm{d}\tilde{t}\\
				&+(\lambda+\mu)\int_{0}^{t}\int_{\Omega}\Delta \phi(\partial_{\tilde{t}}\nabla\cdot{\rm\mathbf{curl}}\psi)\mathrm{d}\boldsymbol{x}\mathrm{d}\tilde{t}\\
				&-(\lambda+\mu)\int_{0}^{t}\int_{\partial B_{\tilde{R}}}\left[ \Delta\phi\nu_{\partial B_{\tilde{R}}}\right] \cdot\partial_{\tilde{t}}{\rm \mathbf{curl}}\psi\mathrm{d}s\mathrm{d}\tilde{t}+(\lambda+\mu)\int_{0}^{t}\int_{\partial D}\left[ \Delta\phi\nu_{\partial D}\right] \cdot\partial_{\tilde{t}}{\rm \mathbf{curl}}\psi\mathrm{d}s\mathrm{d}\tilde{t}\\
				=&(\lambda+\mu)\int_{\Omega}\int_{0}^{t}\frac{1}{2}\frac{\mathrm{d}}{\mathrm{d}\tilde{t}}\left[ \left( \Delta\phi\right) ^2\right]\mathrm{d}\tilde{t} \mathrm{d}\boldsymbol{x}-(\lambda+\mu)\int_{0}^{t}\int_{\partial B_{\tilde{R}}}\left[ \Delta\phi\nu_{\partial B_{\tilde{R}}}\right] \cdot\left[ \partial_{\tilde{t}}\nabla\phi+\partial_{\tilde{t}}{\rm \mathbf{curl}}\psi\right] \mathrm{d}s\mathrm{d}\tilde{t}\\
				&+(\lambda+\mu)\int_{0}^{t}\int_{\partial D}\left[ \Delta\phi\nu_{\partial D}\right] \cdot\left[ \partial_{\tilde{t}}\nabla\phi+\partial_{\tilde{t}}{\rm \mathbf{curl}}\psi\right] \mathrm{d}s\mathrm{d}\tilde{t}\\
				=&\frac{\lambda+\mu}{2}\left\| \Delta\phi\right\| _{L^2(\Omega)}^2-(\lambda+\mu)\int_{0}^{t}\int_{\partial B_{\tilde{R}}}\left[ \Delta\phi\nu_{\partial B_{\tilde{R}}}\right] \cdot\left[ \partial_{\tilde{t}}\nabla\phi+\partial_{\tilde{t}}{\rm \mathbf{curl}}\psi\right] \mathrm{d}s\mathrm{d}\tilde{t}\\
				&+(\lambda+\mu)\int_{0}^{t}\int_{\partial D}\left[ \Delta\phi\nu_{\partial D}\right] \cdot\left[ \partial_{\tilde{t}}\nabla\phi+\partial_{\tilde{t}}{\rm \mathbf{curl}}\psi\right] \mathrm{d}s\mathrm{d}\tilde{t},
			\end{aligned}
		\end{equation*}
		\begin{equation*}
			\begin{aligned}
				A_3=&\mu\int_{0}^{t}\int_{\Omega}\left( \nabla\nabla\phi\right) :\left( \partial_{\tilde{t}}\nabla\nabla\phi\right)\mathrm{d}\boldsymbol{x}\mathrm{d}\tilde{t}\\
				&-\mu \int_{0}^{t}\int_{\partial B_{\tilde{R}}}\left[ \left( \nabla\nabla\phi\right) \nu_{\partial B_{\tilde{R}}}\right] \cdot\partial_{\tilde{t}}\nabla\phi\mathrm{d}s\mathrm{d}\tilde{t}+\mu \int_{0}^{t}\int_{\partial D}\left[ \left( \nabla\nabla\phi\right) \nu_{\partial D}\right] \cdot\partial_{\tilde{t}}\nabla\phi\mathrm{d}s\mathrm{d}\tilde{t}\\
				&+\mu\int_{0}^{t}\int_{\Omega}\left( \nabla\nabla\phi\right) :\left( \partial_{\tilde{t}}\nabla{\rm \mathbf{curl}}\psi\right)\mathrm{d}\boldsymbol{x}\mathrm{d}\tilde{t}\\
				&-\mu \int_{0}^{t}\int_{\partial B_{\tilde{R}}}\left[ \left( \nabla\nabla\phi\right) \nu_{\partial B_{\tilde{R}}}\right] \cdot\partial_{\tilde{t}}{\rm \mathbf{curl}}\psi\mathrm{d}s\mathrm{d}\tilde{t}+\mu \int_{0}^{t}\int_{\partial D}\left[ \left( \nabla\nabla\phi\right) \nu_{\partial D}\right] \cdot\partial_{\tilde{t}}{\rm \mathbf{curl}}\psi\mathrm{d}s\mathrm{d}\tilde{t}\\
				&+\mu\int_{0}^{t}\int_{\Omega}\left( \nabla{\rm \mathbf{curl}}\psi\right) :\left( \partial_{\tilde{t}}\nabla\nabla\phi\right)\mathrm{d}\boldsymbol{x}\mathrm{d}\tilde{t}\\
				&-\mu \int_{0}^{t}\int_{\partial B_{\tilde{R}}}\left[ \left( \nabla{\rm \mathbf{curl}}\psi\right) \nu_{\partial B_{\tilde{R}}}\right] \cdot\partial_{\tilde{t}}\nabla\phi\mathrm{d}s\mathrm{d}\tilde{t}+\mu \int_{0}^{t}\int_{\partial D}\left[ \left( \nabla{\rm \mathbf{curl}}\psi\right) \nu_{\partial D}\right] \cdot\partial_{\tilde{t}}\nabla\phi\mathrm{d}s\mathrm{d}\tilde{t}\\
				&+\mu\int_{0}^{t}\int_{\Omega}\left( \nabla{\rm \mathbf{curl}}\psi\right) :\left( \partial_{\tilde{t}}\nabla{\rm \mathbf{curl}}\psi\right)\mathrm{d}\boldsymbol{x}\mathrm{d}\tilde{t}\\
				&-\mu \int_{0}^{t}\int_{\partial B_{\tilde{R}}}\left[ \left( \nabla{\rm \mathbf{curl}}\psi\right) \nu_{\partial B_{\tilde{R}}}\right] \cdot\partial_{\tilde{t}}{\rm \mathbf{curl}}\psi\mathrm{d}s\mathrm{d}\tilde{t}+\mu \int_{0}^{t}\int_{\partial D}\left[ \left( \nabla{\rm \mathbf{curl}}\psi\right) \nu_{\partial D}\right] \cdot\partial_{\tilde{t}}{\rm \mathbf{curl}}\psi\mathrm{d}s\mathrm{d}\tilde{t}\\
				=&\mu\int_{\Omega}\int_{0}^{t}\left[ \frac{1}{2}\frac{\mathrm{d}}{\mathrm{d}\tilde{t}}\left( \nabla\nabla\phi:\nabla\nabla\phi\right) +\frac{\mathrm{d}}{\mathrm{d}\tilde{t}}\left( \nabla\nabla\phi:\nabla{\rm \mathbf{curl}}\psi\right) +\frac{1}{2}\frac{\mathrm{d}}{\mathrm{d}\tilde{t}}\left( \nabla{\rm \mathbf{curl}}\psi:\nabla{\rm \mathbf{curl}}\psi\right) \right] \mathrm{d}\tilde{t}\mathrm{d}\boldsymbol{x}\\
				&-\mu\int_{0}^{t}\int_{\partial B_{\tilde{R}}}\left[ \left( \nabla\nabla\phi\right) \nu_{\partial B_{\tilde{R}}}+\left( \nabla{\rm \mathbf{curl}}\psi\right) \nu_{\partial B_{\tilde{R}}}\right] \cdot\left[ \partial_{\tilde{t}}\nabla\phi+\partial_{\tilde{t}}{\rm\mathbf{curl}}\psi\right] \mathrm{d}s\mathrm{d}\tilde{t}\\
				&+\mu\int_{0}^{t}\int_{\partial D}\left[ \left( \nabla\nabla\phi\right) \nu_{\partial D}+\left( \nabla{\rm \mathbf{curl}}\psi\right) \nu_{\partial D}\right] \cdot\left[ \partial_{\tilde{t}}\nabla\phi+\partial_{\tilde{t}}{\rm\mathbf{curl}}\psi\right] \mathrm{d}s\mathrm{d}\tilde{t}\\
			\end{aligned}
			\end{equation*}
		\begin{equation*}
			\begin{aligned}
				\qquad=&\mu\int_{\Omega}\frac{1}{2}\int_{0}^{t} \frac{\mathrm{d}}{\mathrm{d}\tilde{t}}\left[ \left( \nabla\nabla\phi+\nabla{\rm \mathbf{curl}}\psi\right):\left( \nabla\nabla\phi+\nabla{\rm \mathbf{curl}}\psi\right)\right] \mathrm{d}\tilde{t}\mathrm{d}\boldsymbol{x}\\
				&-\mu\int_{0}^{t}\int_{\partial B_{\tilde{R}}}\left[ \left( \nabla\nabla\phi\right) \nu_{\partial B_{\tilde{R}}}+\left( \nabla{\rm \mathbf{curl}}\psi\right) \nu_{\partial B_{\tilde{R}}}\right] \cdot\left[ \partial_{\tilde{t}}\nabla\phi+\partial_{\tilde{t}}{\rm\mathbf{curl}}\psi\right] \mathrm{d}s\mathrm{d}\tilde{t}\\
				&+\mu\int_{0}^{t}\int_{\partial D}\left[ \left( \nabla\nabla\phi\right) \nu_{\partial D}+\left( \nabla{\rm \mathbf{curl}}\psi\right) \nu_{\partial D}\right] \cdot\left[ \partial_{\tilde{t}}\nabla\phi+\partial_{\tilde{t}}{\rm\mathbf{curl}}\psi\right] \mathrm{d}s\mathrm{d}\tilde{t}\\
				=&\frac{\mu}{2}\left\|  \nabla\nabla\phi+\nabla{\rm \mathbf{curl}}\psi\right\| _{L^2(\Omega)^{2\times 2}}^2-\mu\int_{0}^{t}\int_{\partial B_{\tilde{R}}}\left[ \left( \nabla\nabla\phi\right) \nu_{\partial B_{\tilde{R}}}+\left( \nabla{\rm \mathbf{curl}}\psi\right) \nu_{\partial B_{\tilde{R}}}\right] \cdot\left[ \partial_{\tilde{t}}\nabla\phi+\partial_{\tilde{t}}{\rm\mathbf{curl}}\psi\right] \mathrm{d}s\mathrm{d}\tilde{t}\\
				&+\mu\int_{0}^{t}\int_{\partial D}\left[ \left( \nabla\nabla\phi\right) \nu_{\partial D}+\left( \nabla{\rm \mathbf{curl}}\psi\right) \nu_{\partial D}\right] \cdot\left[ \partial_{\tilde{t}}\nabla\phi+\partial_{\tilde{t}}{\rm\mathbf{curl}}\psi\right] \mathrm{d}s\mathrm{d}\tilde{t},
			\end{aligned}
		\end{equation*}
		where $A:B={\rm tr}(AB^\top)$ is the Frobenius inner product of square matrices $A$ and $B$. The coupled boundary condition on $\partial D$ and the zero boundary condition on $\partial B_{\tilde{R}}$ in \eqref{initial-boundary value problem} yield that
		\begin{equation*}
			\begin{aligned}
				&A_2=\frac{\lambda+\mu}{2}\left\| \Delta\phi\right\| _{L^2(\Omega)}^2,\\
				&A_3=\frac{\mu}{2}\left\|  \nabla\nabla\phi+\nabla{\rm \mathbf{curl}}\psi\right\| _{L^2(\Omega)^{2\times 2}}^2.	
			\end{aligned}
		\end{equation*} 
		So we have
		\begin{equation*}
			\begin{aligned}
				0=&A_1+A_2+A_3\\
				=&\frac{1}{2}\left\| \partial_{t}\nabla\phi+\partial_{t}{\rm\mathbf{curl}}\psi\right\| _{L^2(\Omega)^2}^2+\frac{\lambda+\mu}{2}\left\| \Delta\phi\right\| _{L^2(\Omega)}^2+\frac{\mu}{2}\left\|  \nabla\nabla\phi+\nabla{\rm \mathbf{curl}}\psi\right\| _{L^2(\Omega)^{2\times 2}}^2,
			\end{aligned}
		\end{equation*}
		From this, we conclude that
		\begin{equation*}
			\nabla\phi+{\rm \mathbf{curl}}\psi=0,\quad \Delta\phi=0,\quad {\rm in}~\Omega,
		\end{equation*}
		which implies that $\partial_{t}^2\phi=0$ and $\partial_{t}^2\psi=0$. It follows from Cauchy-Schwarz inequality and Young's inequality that
		\begin{equation*}
			\begin{aligned}
				\left\| \phi(t)\right\| _{L^2(\Omega)}^2=&\int_{0}^{t}\frac{\mathrm{d}}{\mathrm{d}\tau}\left\| \phi(\tau)\right\| _{L^2(\Omega)}^2\mathrm{d}\tau=2\int_{0}^{t}\int_{\Omega}\phi(\tau)\partial_{\tau}\phi(\tau)\mathrm{d}\boldsymbol{x}\mathrm{d}\tau\\
				&\le\int_{0}^{t}2\left\| \phi(\tau)\right\| _{L^2(\Omega)}\left\| \partial_{\tau}\phi(\tau)\right\| _{L^2(\Omega)}\mathrm{d}\tau\\
				&\le\int_{0}^{t}\left[ \epsilon\left\| \phi(\tau)\right\| _{L^2(\Omega)}^2+\frac{1}{4\epsilon}\left\| \partial_{\tau}\phi(\tau)\right\| _{L^2(\Omega)}^2\right] \mathrm{d}\tau\\
				&\le\epsilon T\max_{t\in(0,T]}\left\| \phi(t)\right\| _{L^2(\Omega)}^2+\frac{T}{4\epsilon}\max_{t\in(0,T]}\left\| \partial_{t}\phi(t)\right\| _{L^2(\Omega)}^2,
			\end{aligned}
		\end{equation*}
		that means 
		\[
		\max_{t\in(0,T]}\left\| \phi(t)\right\| _{L^2(\Omega)}^2
		\le\epsilon T\max_{t\in(0,T]}\left\| \phi(t)\right\| _{L^2(\Omega)}^2+\frac{T}{4\epsilon}\max_{t\in(0,T]}\left\| \partial_{t}\phi(t)\right\| _{L^2(\Omega)}^2.
		\]
			Taking $\epsilon=\frac{1}{2T}$, we have
		\begin{equation*}
			\max_{t\in(0,T]}\left\| \phi(t)\right\| _{L^2(\Omega)}^2\le T^2\max_{t\in(0,T]}\left\| \partial_{t}\phi(t)\right\| _{L^2(\Omega)}^2.
		\end{equation*}
		Analogously, we can obtain
		\begin{equation*}
			\max_{t\in(0,T]}\left\| \partial_{t}\phi(t)\right\| _{L^2(\Omega)}^2\le T^2\max_{t\in(0,T]}\left\| \partial_{t}^2\phi(t)\right\| _{L^2(\Omega)}^2.
		\end{equation*}
		Since $\partial_{t}^2\phi=0$ in $\Omega$ for $t\in[0,T]$, we conclude that $\phi=0$ in $\Omega$. Similarly, we can get $\psi=0$ in $\Omega$. The proof is completed by noting that $C^\infty(0,T;H^2(\Omega))$ is dense in $L^2(0,T;H^2(\Omega))$.
	\end{proof}
	
	The inverse obstacle scattering problem for elastic waves in the time domain can be stated as following:
	\begin{problem}\label{problem1}
		Given a time domain incident plane wave $\boldsymbol{u}^{\rm inc}$ for fixed Lam$\acute{\rm e}$ parameters $\lambda$, $\mu$ and a single incident direction $\boldsymbol{d}$, the inverse obstacle scattering problem is to determine the boundary $\partial D$ from the scattered field data $\boldsymbol{v}(x,t)$, $x\in \partial B_R$, $t\in[0,T]$.
	\end{problem}

	\section{Boundary integral equations}
	\subsection{Retarded potential boundary integral equation method}
	We will show the mathematical expression of the retarded potential boundary integral equation method in this section. For the initial-boundary value problem \eqref{initial-boundary value problem}, the retarded single layer potential is defined as
	\[
	(SL_{\partial D}g)(\boldsymbol{x},t):=\int_{0}^{t}\int_{\partial D} k(t-\tau,|\boldsymbol{x}-\boldsymbol{y}|;c)g(\boldsymbol{y},\tau)\mathrm{d}s_{\boldsymbol{y} }\mathrm{d}\tau,\quad t\in(0,\infty),~\boldsymbol{x}\in\mathbb{R}^2\setminus\partial D,
	\]
	where 
	\[
	k(t,r):=\frac{H(t-c^{-1}r)}{2\pi\sqrt{t^2-c^{-2}r^2}}
	\]
	is the fundamental solution and $H$ is the Heaviside function. Then the corresponding single layer operator is denoted by
	\[
	(S_{\partial D}g)(\boldsymbol{x},t):=\int_{0}^{t}\int_{\partial D}k(t-\tau,|\boldsymbol{x}-\boldsymbol{y}|;c)g(\boldsymbol{y},\tau)\mathrm{d}s_{\boldsymbol{y} }\mathrm{d}\tau,\quad t\in(0,\infty),~\boldsymbol{x}\in \partial D.
	\]
	The observation data are given by the field measured on
	a known curve $\partial B_R$, surrounding the unknown obstacle $D$ over a finite time interval $(0,T]$. We choose the terminal time $T$ such that the energy of the scattered data inside the interested domain is negligible when $t>T$. Furthermore, we assume that the solution of \eqref{initial-boundary value problem} is given as single layer potentials with density $g_1$, $g_2$:
	\begin{equation}\label{single-layer}
		\begin{aligned}
			&\phi=\int_{0}^{t}\int_{\partial D} k(t-\tau,|\boldsymbol{x}-\boldsymbol{y}|;c_1)g_1(\boldsymbol{y},\tau)\mathrm{d}s_{\boldsymbol{y}}\mathrm{d}\tau,\quad t\in(0,T],~\boldsymbol{x}\in \mathbb{R}^2\setminus\overline{D},\\
			&\psi=\int_{0}^{t}\int_{\partial D} k(t-\tau,|\boldsymbol{x}-\boldsymbol{y}|;c_2)g_2(\boldsymbol{y},\tau)\mathrm{d}s_{\boldsymbol{y}}\mathrm{d}\tau,\quad t\in(0,T],~\boldsymbol{x}\in \mathbb{R}^2\setminus\overline{D}.
		\end{aligned}
	\end{equation}
It is easy to verify that $\phi$ and $\psi$ satisfy the wave equation and the initial conditions in \eqref{initial-boundary value problem}.
Letting $\boldsymbol{x}\in \mathbb{R}^2\setminus\overline{D}$ approach the boundary $\partial D$ in \eqref{single-layer}, and using the jump relation of single-layer potentials and the boundary condition of \eqref{initial-boundary value problem}, we deduce for $\boldsymbol{x}\in\partial D$, $t\in(0,T]$ that
	\begin{equation}\label{boundary condition12}
		\begin{aligned}
			-\frac{1}{2}&g_1(\boldsymbol{x},t)+\int_{0}^{t}\int_{\partial D}\frac{\partial k(t-\tau,|\boldsymbol{x}-\boldsymbol{y}|;c_1)}{\partial \nu(\boldsymbol{x})}g_1(\boldsymbol{y},\tau)\mathrm{d}s_{\boldsymbol{y}}\mathrm{d}\tau\\
			&+\int_{0}^{t}\int_{\partial D}\frac{\partial k(t-\tau,|\boldsymbol{x}-\boldsymbol{y}|;c_2)}{\partial \tau(\boldsymbol{x})}g_2(\boldsymbol{y},\tau)\mathrm{d}s_{\boldsymbol{y}}\mathrm{d}\tau=f_1(\boldsymbol{x},t),\\
			&\int_{0}^{t}\int_{\partial D}\frac{\partial k(t-\tau,|\boldsymbol{x}-\boldsymbol{y}|;c_1)}{\partial \tau(\boldsymbol{x})}g_1(\boldsymbol{y},\tau)\mathrm{d}s_{\boldsymbol{y}}\mathrm{d}\tau\\
			+\frac{1}{2}&g_2(\boldsymbol{x},t)-\int_{0}^{t}\int_{\partial D}\frac{\partial k(t-\tau,|\boldsymbol{x}-\boldsymbol{y}|;c_2)}{\partial \nu(\boldsymbol{x})}g_2(\boldsymbol{y},\tau)\mathrm{d}s_{\boldsymbol{y}}\mathrm{d}\tau=f_2(\boldsymbol{x},t),
		\end{aligned}
	\end{equation}
where $f_1(\boldsymbol{x},t)=-\nu(\boldsymbol{x})\cdot\boldsymbol{u}^{\rm inc}(\boldsymbol{x},t)$ and $f_2(\boldsymbol{x},t)=-\tau(\boldsymbol{x})\cdot\boldsymbol{u}^{\rm inc}(\boldsymbol{x},t)$.
Now we discuss the uniqueness result for \eqref{boundary condition12}.
\begin{theorem}
	There exists at most one solution to the boundary integral equations \eqref{boundary condition12}.
\end{theorem}
\begin{proof}
	It suffices to show that $g_1=g_2=0$ if $f_1=f_2=0$. For $\boldsymbol{x}\in\mathbb{R}^2\setminus\partial D$, $t\in[0,T]$, we define single layer potentials
		\begin{align*}
			&\phi(\boldsymbol{x},t)=\int_{0}^{t}\int_{\partial D}k(|\boldsymbol{x}-\boldsymbol{y}|,t-\tau;c_1)g_1(\boldsymbol{y},\tau)\mathrm{d}s_{\boldsymbol{y}}\mathrm{d}\tau,\\
			&\psi(\boldsymbol{x},t)=\int_{0}^{t}\int_{\partial D}k(|\boldsymbol{x}-\boldsymbol{y}|,t-\tau;c_2)g_2(\boldsymbol{y},\tau)\mathrm{d}s_{\boldsymbol{y}}\mathrm{d}\tau.
		\end{align*}
		Let
		\begin{equation*}
			\phi(\boldsymbol{x},t)=\begin{cases}
				\phi_i,\quad \boldsymbol{x}\in D,\\
				\phi_e,\quad \boldsymbol{x}\in \mathbb{R}^2\setminus\overline{D},
			\end{cases}\quad \psi(\boldsymbol{x},t)=\begin{cases}
				\psi_i,\quad \boldsymbol{x}\in D,\\
				\psi_e,\quad \boldsymbol{x}\in \mathbb{R}^2\setminus\overline{D}.
			\end{cases}
		\end{equation*}
	Then $(\phi_e,\psi_e)$ satisfies \eqref{initial-boundary value problem} with $f_1=0$, $f_2=0$. By the uniqueness result in Theorem \ref{problemuniqueness}, it holds that
\[
\phi_e(\boldsymbol{x},t)=\psi_e(\boldsymbol{x},t)=0,\quad \boldsymbol{x}\in\mathbb{R}^2\setminus\overline{D},~t\in[0,T].
\]
Using the jump condition of single layer potentials, we have on $\partial D$ that
\begin{equation}\label{jump1}
	\phi_e-\phi_i=0,\quad \psi_e-\psi_i=0,
\end{equation}
\begin{equation}\label{jump2}
	\partial_{\nu}\phi_e-\partial_{\nu}\phi_i=-g_1,\quad \partial_{\nu}\psi_e-\partial_{\nu}\psi_i=-g_2.
\end{equation}
Combining \eqref{jump1} and the fact $\phi_e=\psi_e=0$ in $\mathbb{R}^2\setminus\overline{D}$, we derive that $\phi_i$ and $\psi_i$ satisfy the zero boundary condition on $\partial D$. By the uniqueness of the interior problem for the wave equation, it holds that $\phi_i=\psi_i=0$ in $D$. We conclude that $g_1=g_2=0$ by \eqref{jump2}, which completes the proof.
\end{proof}
\subsection{Convolution quadrature}
For time discretization, we adopt the convolution quadrature method \cite{CQ,sayas} to deal with the boundary integral equation \eqref{boundary condition12}. We write 
\[
h_1(t-\tau,|\boldsymbol{x}-\boldsymbol{y}|;c)=\frac{\partial k(t-\tau,|\boldsymbol{x}-\boldsymbol{y}|;c)}{\partial \nu(\boldsymbol{x})},
\]
\[
h_2(t-\tau,|\boldsymbol{x}-\boldsymbol{y}|;c)=\frac{\partial k(t-\tau,|\boldsymbol{x}-\boldsymbol{y}|;c)}{\partial \tau(\boldsymbol{x})}.
\]
For simplicity, the left hands of \eqref{boundary condition12} can be written as
	\begin{equation}\label{convolution equation}
		\begin{aligned}
			-\frac{1}{2}g_1(\boldsymbol{x},t)+w_1(c_1)*g_1+w_2(c_2)*g_2&=-\nu(\boldsymbol{x})\cdot\boldsymbol{u}^{\rm inc}(\boldsymbol{x},t),\\
			w_2(c_1)*g_1+\frac{1}{2}g_2(\boldsymbol{x},t)-w_1(c_2)*g_2&=-\tau(\boldsymbol{x})\cdot\boldsymbol{u}^{\rm inc}(\boldsymbol{x},t),
		\end{aligned}
	\end{equation}
where we have set
\[
(w_j(c)*g)(\boldsymbol{x},t):=\int_{0}^{t}\Big(w_j(t-\tau,c)g(\tau)\Big)(\boldsymbol{x})\mathrm{d}\tau,\quad \boldsymbol{x}\in\partial D,
\]
and $w_j(t-\tau,c)g(\tau)$ is a parameter-dependent integral operator described by
\[
\Big(w_j(t-\tau,c)g(\tau)\Big)(\boldsymbol{x}):=\int_{\partial D}h_j(t-\tau,|\boldsymbol{x}-\boldsymbol{y}|;c)g(\boldsymbol{y},\tau)\mathrm{d}s_{\boldsymbol{y}}, \quad \boldsymbol{x}\in\partial D.
\]
Define the Laplace transform of $w_j$, $j=1,2$ by
\[
W_j(s,c)=\int_{0}^{\infty}w_j(t,c)e^{-st}\mathrm{d}t.
\] 
Then the specific form of $W_j(s)$ is given by
	\begin{equation*}
		W_j(s,c)G(\boldsymbol{x})=\left\{
		\begin{aligned}
			&\int_{\partial D} \frac{\partial K(|\boldsymbol{x}-\boldsymbol{y}|;c,s)}{\partial \nu(\boldsymbol{x})}G(\boldsymbol{y})\mathrm{d}s_{\boldsymbol{y}}, &&j=1,\\
			&\int_{\partial D}\frac{\partial K(|\boldsymbol{x}-\boldsymbol{y}|;c,s)}{\partial \tau(\boldsymbol{x})}G(\boldsymbol{y})\mathrm{d}s_{\boldsymbol{y}}, &&j=2,\\
		\end{aligned}
		\right.
	\end{equation*}
where $K(|\boldsymbol{x}-\boldsymbol{y}|;c,s)=\frac{\mathrm{i}}{4}H_0^{(1)}(\mathrm{i}\frac{s}{c}|\boldsymbol{x}-\boldsymbol{y}|)$, and $H_0^{(1)}$ is the Hankel function of the first kind with zero-order. Since the time discretization is implemented over $[0,T]$, we divided this interval into $N+1$ time steps with
\[
t_n=n\Delta t,\quad n=0,1,\ldots,N,~\Delta t=T/N.
\]
Applying the unconditionally backward difference formula of second order scheme as in \cite{sayas}, the convolution $w_j(c)*g$ in equation \eqref{convolution equation} can be written as 
\begin{equation}\label{convolution}
	(w_j(c)*g)(t_n)=\sum_{l=0}^{n}\omega_{n-l}^{\Delta t}(W_j(c))g^{\Delta t}(t_l),
\end{equation}
where the convolution weights $\omega_n^{\Delta t}(W_j(c))$ are implicitly defined by
\[
W_j(\gamma(\zeta)/\Delta t,c)=\sum_{n=0}^{\infty}\omega_{n-l}^{\Delta t}(W_j(c))\zeta^n,\quad |\zeta|<1,
\]
with $ \gamma(\zeta)=\frac{1}{2}(\zeta^2-4\zeta+3)$. Here, $\omega_{l}^{\Delta t}(W_j(c))$ can be numerically computed by
\[
\omega_{l}^{\Delta t}(W_j(c)):=\frac{1}{2\pi\mathrm{i}}\oint_C \frac{W_j(\gamma(\zeta)/\Delta t,c)}{\zeta^{l+1}}\mathrm{d}\zeta,
\]
where $C$ is chosen as a circle centered at the origin with radius $\tilde{\lambda}<1$. Employing the trapezoidal rule, the approximate convolution weights are then given by
\begin{equation}\label{omega}
	\omega_{l}^{\Delta t,\tilde{\lambda}}(W_j(c)):=\frac{\tilde{\lambda}^{-l}}{N+1}\sum_{k=0}^{N}W_j(s_k,c)\zeta_{N+1}^{lk},
\end{equation}
where $\zeta_{N+1}=e^{\frac{2\pi i}{N+1}}$, $s_k=\frac{\gamma(\tilde{\lambda} \zeta_{N+1}^{-k})}{\Delta t}$.
Substituting the approximate weights \eqref{omega} into \eqref{convolution}, the equation \eqref{convolution equation} can be replaced by the time-discrete problem: find $g_{j,k}(\cdot)=g_j(\cdot,t_k)$, such that for $n=0,\ldots,N$,
	\begin{equation}\label{discrete problem}
		\begin{aligned}
			-\frac{1}{2}g_{1,n}(\boldsymbol{x})+\sum_{l=0}^{n}(\omega_{n-l}^{\Delta t,\tilde{\lambda}}(W_1(c_1))g_{1,l})(\boldsymbol{x})+\sum_{l=0}^{n}(\omega_{n-l}^{\Delta t,\tilde{\lambda}}(W_2(c_2))g_{2,l})(\boldsymbol{x})&=-\nu(\boldsymbol{x})\cdot\boldsymbol{u}_n^{\rm inc}(\boldsymbol{x}),\\
			\sum_{l=0}^{n}(\omega_{n-l}^{\Delta t,\tilde{\lambda}}(W_2(c_1))g_{1,l})(\boldsymbol{x})+\frac{1}{2}g_{2,n}(\boldsymbol{x})-\sum_{l=0}^{n}(\omega_{n-l}^{\Delta t,\tilde{\lambda}}(W_1(c_2))g_{2,l})(\boldsymbol{x})&=-\tau(\boldsymbol{x})\cdot\boldsymbol{u}_n^{\rm inc}(\boldsymbol{x}).
		\end{aligned}
	\end{equation}
where $\boldsymbol{u}_n^{\rm inc}(\cdot)=\boldsymbol{u}^{\rm inc}(\cdot,t_n)$. From Cauchy's theorem, it follows that $\omega_l^{\Delta t}(W_j)=0$ for $j=1,2$, $l<0$ \cite{sayas}. Then \eqref{discrete problem} is equivalent to 
	\begin{equation*}\label{discrete problem N}
		\begin{aligned}
			-\frac{1}{2}g_{1,n}(\boldsymbol{x})+\sum_{l=0}^{N}(\omega_{n-l}^{\Delta t}(W_1(c_1))g_{1,l})(\boldsymbol{x})+\sum_{l=0}^{N}(\omega_{n-l}^{\Delta t}(W_2(c_2))g_{2,l})(\boldsymbol{x})&=-\nu(\boldsymbol{x})\cdot\boldsymbol{u}_n^{\rm inc}(\boldsymbol{x}),\\
			\sum_{l=0}^{N}(\omega_{n-l}^{\Delta t}(W_2(c_1))g_{1,l})(\boldsymbol{x})+\frac{1}{2}g_{2,n}(\boldsymbol{x})-\sum_{l=0}^{N}(\omega_{n-l}^{\Delta t}(W_1(c_2))g_{2,l})(\boldsymbol{x})&=-\tau(\boldsymbol{x})\cdot\boldsymbol{u}_n^{\rm inc}(\boldsymbol{x}),
		\end{aligned}
	\end{equation*}
for $n=0,\ldots,N$.
By using \eqref{omega}, we can finally get the following decoupled boundary integral equations for $l=0,\ldots N$, i.e.
	\begin{equation}\label{field equation}
		\begin{aligned}
			-\frac{1}{2}\hat{g}_{1,l}(\boldsymbol{x})+(W_1(s_l,c_1)\hat{g}_{1,l})(\boldsymbol{x})+(W_2(s_l,c_2)\hat{g}_{2,l})(\boldsymbol{x})&=-\nu(\boldsymbol{x})\cdot\hat{\boldsymbol{u}}_l^{\rm inc}(\boldsymbol{x}),\\
			(W_2(s_l,c_1)\hat{g}_{1,l})(\boldsymbol{x})+\frac{1}{2}\hat{g}_{2,l}(\boldsymbol{x})-(W_1(s_l,c_2)\hat{g}_{2,l})(\boldsymbol{x})&=-\tau(\boldsymbol{x})\cdot\hat{\boldsymbol{u}}_l^{\rm inc}(\boldsymbol{x}).
		\end{aligned}
	\end{equation}
where $\hat{g}_{j,l}$, $j=1,2$ and $\hat{u}^{\rm inc}_l$ are the scaled discrete Fourier transform, i.e.
\[
\hat{g}_{j,l}:=\sum_{k=0}^{N}\tilde{\lambda}^k g_{j,k}\zeta_{N+1}^{-lk},\quad \hat{\boldsymbol{u}}_l^{\rm inc}:=\sum_{n=0}^{N}\tilde{\lambda}^n \boldsymbol{u}_n^{\rm inc}\zeta_{N+1}^{-ln}.
\]
Consequently, we can obtain the density $g_j$ by using the inverse transform:
\[
g_{j,n}=\frac{\tilde{\lambda}^{-n}}{N+1}\sum_{l=0}^{N}\hat{g}_{j,l}\zeta_{N+1}^{nl}.
\]
For the scattered field, by using $\boldsymbol{v}=\nabla \phi+{\rm \mathbf{curl}}\psi$, we have
	\begin{equation*}\label{scattered field integral equation}
		\begin{aligned}
			\boldsymbol{v}(\boldsymbol{x},t)=&\nabla_{\boldsymbol{x}}\int_{0}^{t}\int_{\partial D} k(t-\tau,|\boldsymbol{x}-\boldsymbol{y}|;c_1)g_1(\boldsymbol{y},\tau)\mathrm{d}s_{\boldsymbol{y}}\mathrm{d}\tau\\
			&+{\rm \mathbf{curl}}_{\boldsymbol{x}}\int_{0}^{t}\int_{\partial D} k(t-\tau,|\boldsymbol{x}-\boldsymbol{y}|;c_2)g_2(\boldsymbol{y},\tau)\mathrm{d}s_{\boldsymbol{y}}\mathrm{d}\tau.
		\end{aligned}
	\end{equation*}
Similar to the process of solving \eqref{convolution equation}, the scaled discrete Fourier transform of $\boldsymbol{v}_n(\cdot)=\boldsymbol{v}(\cdot,t_n)$ can be represented as
	\begin{equation}\label{3.12}
		\hat{\boldsymbol{v}}_l(\boldsymbol{x})=\int_{\partial D}\nabla_{\boldsymbol{x}}K(|\boldsymbol{x}-\boldsymbol{y}|;c_1,s_l)\hat{g}_{1,l}(\boldsymbol{y})\mathrm{d}s_{\boldsymbol{y}}+\int_{\partial D}\mathbf{curl}_{\boldsymbol{x}}K(|\boldsymbol{x}-\boldsymbol{y}|;c_2,s_l)\hat{g}_{2,l}(\boldsymbol{y})\mathrm{d}s_{\boldsymbol{y}},
	\end{equation}
for $l=0,\ldots,N$, $\boldsymbol{x}\in \mathbb{R}^2\setminus\overline{D}$, and $\hat{g}_{j,l}$, $j=1,2$ can be obtained from the equation \eqref{field equation}, the scattering wave $\boldsymbol{v}_n$ is given by
\[
\boldsymbol{v}_n=\frac{\tilde{\lambda}^{-n}}{N+1}\sum_{l=0}^{N}\hat{\boldsymbol{v}}_l\zeta_{N+1}^{nl},\quad n=0,\ldots,N.
\]
For the detail of the convolution quadrature method we refer to \cite{CQ,sayas}. 
\section{Nystr$\mathrm{\ddot o}$m-type discretization for boundary integral equations}
In this section, we present a Nystr$\mathrm{\ddot o}$m-type discretization of the equations \eqref{field equation}. For convenience, we denote the normal and tangential derivative boundary integral operators by
	\begin{align*}
		(D_s^cg)(\boldsymbol{x})&=2\int_{\partial D} \frac{\partial K(|\boldsymbol{x}-\boldsymbol{y}|;c,s)}{\partial \nu(\boldsymbol{x})}g(\boldsymbol{y})\mathrm{d}s_{\boldsymbol{y}},\quad \boldsymbol{x}\in\partial D,\\
		(H_s^cg)(\boldsymbol{x})&=2\int_{\partial D} \frac{\partial K(|\boldsymbol{x}-\boldsymbol{y}|;c,s)}{\partial \tau(\boldsymbol{x})}g(\boldsymbol{y})\mathrm{d}s_{\boldsymbol{y}},\quad \boldsymbol{x}\in\partial D.	
	\end{align*}
In addition, we need to introduce the vector boundary integral operators
	\begin{align*}
	(\boldsymbol{N}_s^cg)(\boldsymbol{x})&=\int_{\partial D}\nabla_{\boldsymbol{x}}K(|\boldsymbol{x}-\boldsymbol{y}|;c,s)g(\boldsymbol{y})\mathrm{d}s_{\boldsymbol{y}},\quad \boldsymbol{x}\in\partial B_R,\\
	(\boldsymbol{T}_s^cg)(\boldsymbol{x})&=\int_{\partial D}\mathbf{curl}_{\boldsymbol{x}}K(|\boldsymbol{x}-\boldsymbol{y}|;c,s)g(\boldsymbol{y})\mathrm{d}s_{\boldsymbol{y}},\quad \boldsymbol{x}\in\partial B_R.
	\end{align*}
Then, the boundary integral equations \eqref{field equation} can be rewritten in the operator form
	\begin{equation}\label{4.1}
		-\hat{g}_{1,l}+D_{s_l}^{c_1}\hat{g}_{1,l}+H_{s_l}^{c_2}\hat{g}_{2,l}=2\hat{f}_{1,l},
	\end{equation}
	\begin{equation}\label{4.2}
		H_{s_l}^{c_1}\hat{g}_{1,l}+\hat{g}_{2,l}-D_{s_l}^{c_2}\hat{g}_{2,l}=2\hat{f}_{2,l},
	\end{equation}
	where $\hat{f}_{1,l}=-\nu\cdot\hat{\boldsymbol{u}}_l^{\rm inc}$, $\hat{f}_{2,l}=-\tau\cdot\hat{\boldsymbol{u}}_l^{\rm inc}$, $l=0,\ldots,N$. The corresponding scattered field \eqref{3.12} can be represented as follows
	\begin{equation}\label{4.2.5}
		\hat{\boldsymbol{v}}_l(\boldsymbol{x})=(\boldsymbol{N}_{s_l}^{c_1}\hat{g}_{1,l})(\boldsymbol{x})+(\boldsymbol{T}_{s_l}^{c_2}\hat{g}_{2,l})(\boldsymbol{x}),\quad \boldsymbol{x}\in \partial B_R.
	\end{equation}
	\subsection{Parametrization}
	For simplicity, the boundary $\partial D$ is assumed to be a star-shaped curve with the parametric form
	\[
	\partial D=\left\lbrace p_D(\hat{\boldsymbol{x}})=\boldsymbol{p}+r(\hat{\boldsymbol{x}})\hat{\boldsymbol{x}}:\boldsymbol{p}=(p_1,p_2)^\top,~\hat{\boldsymbol{x}}\in \mathbb{S}\right\rbrace ,
	\]
	where $\mathbb{S}=\left\lbrace \hat{\boldsymbol{x}}(\theta)=(\cos \theta,\sin \theta)^\top:0\le \theta< 2\pi\right\rbrace $. Let $p_D(\theta)$ be the points on $\partial D$, described by $p_D(\theta):=(p_1,p_2)^\top+r(\theta)(\cos \theta,\sin \theta)^\top$ for $0\le \theta<2\pi$. The observation curve is parameterized by $\partial B_R=p_B(\varsigma)=(b_1,b_2)^\top+(R\cos\varsigma,R\sin\varsigma)^\top$ for $0\le\varsigma<2\pi$. We introduce the parametrized integral operators which are still denoted by $D_s^c$, $H_s^c$, $\boldsymbol{N}_s^c$ and $\boldsymbol{T}_s^c$ for convenience, i.e.,
\begin{align*}
	&(D_{s_l}^c(p_D,\varphi_{j,l}))(\theta)=\frac{1}{G_r(\theta)}\int_{0}^{2\pi}\widetilde{D}(\theta,\eta;s_l,c)\varphi_{j,l}(\eta)\mathrm{d}\eta,\\
	&(H_{s_l}^c(p_D,\varphi_{j,l}))(\theta)=\frac{1}{G_r(\theta)}\int_{0}^{2\pi}\widetilde{H}(\theta,\eta;s_l,c)\varphi_{j,l}(\eta)\mathrm{d}\eta,\\
	&(\boldsymbol{N}_{s_l}^c(p_B, p_D,\varphi_{j,l}))(\varsigma)=\int_{0}^{2\pi}\widetilde{\boldsymbol{N}}(\varsigma,\eta;s_l,c)\varphi_{j,l}(\eta)\mathrm{d}\eta,\\
	&(\boldsymbol{T}_{s_l}^c(p_B, p_D,\varphi_{j,l}))(\varsigma)=\int_{0}^{2\pi}\widetilde{\boldsymbol{T}}(\varsigma,\eta;s_l,c)\varphi_{j,l}(\eta)\mathrm{d}\eta,
\end{align*}
where $\varphi_{j,l}(\eta)=G_r(\eta)\hat{g}_{j,l}(p_D(\eta))$, $j=1,2$, $G_r(\eta):=|p_D'(\eta)|=\sqrt{r^2(\eta)+(r'(\eta))^2}$ is the Jacobian of the transform,
    \begin{align*}
	\widetilde{D}(\theta,\eta;s_l,c)&=\frac{s_l}{2c}n(\theta)\cdot[p_D(\theta)-p_D(\eta)]\frac{H_1^{(1)}(i\frac{s_l}{c}|p_D(\theta)-p_D(\eta)|)}{|p_D(\theta)-p_D(\eta)|},\\
	\widetilde{H}(\theta,\eta;s_l,c)&=\frac{s_l}{2c}n(\theta)^\perp\cdot[p_D(\theta)-p_D(\eta)]\frac{H_1^{(1)}(i\frac{s_l}{c}|p_D(\theta)-p_D(\eta)|)}{|p_D(\theta)-p_D(\eta)|},\\
	\widetilde{\boldsymbol{N}}(\varsigma,\eta;s_l,c)&=\frac{s_l}{4c}\frac{H_1^{(1)}(i\frac{s_l}{c}|p_B(\varsigma)-p_D(\eta)|)}{|p_B(\varsigma)-p_D(\eta)|}(p_B(\varsigma)-p_D(\eta)),\\
	\widetilde{\boldsymbol{T}}(\varsigma,\eta;s_l,c)&=\frac{s_l}{4c}\frac{H_1^{(1)}(i\frac{s_l}{c}|p_B(\varsigma)-p_D(\eta)|)}{|p_B(\varsigma)-p_D(\eta)|}\begin{pmatrix}
		p_{B,2}(\varsigma)-p_{D,2}(\eta)\\
		p_{D,1}(\eta)-p_{B,1}(\varsigma)
	\end{pmatrix},
    \end{align*}
	and
	\begin{equation*}
		\begin{aligned}
			&n(\theta)=\nu(p_D(\theta))|p_D'(\theta)|=(p_{D,2}'(\theta),-p_{D,1}'(\theta))^\top,\\
			&n(\theta)^\perp=\tau(p_D(\theta))|p_D'(\theta)|=(p_{D,1}'(\theta),p_{D,2}'(\theta))^\top.
		\end{aligned}
	\end{equation*}
Hence \eqref{4.1}-\eqref{4.2} can be reformulated as the parametrized integral equations
	\begin{equation}\label{4.3}
		-\varphi_{1,l}+(D_{s_l}^{c_1}(p_D,\varphi_{1,l}))G_r+(H_{s_l}^{c_2}(p_D,\varphi_{2,l}))G_r=\omega_{1,l},
	\end{equation}
	\begin{equation}\label{4.4}
		(H_{s_l}^{c_1}(p_D,\varphi_{1,l}))G_r+\varphi_{2,l}-(D_{s_l}^{c_2}(p_D,\varphi_{2,l}))G_r=\omega_{2,l},
	\end{equation}
where $\omega_{j,l}=2(\hat{f}_{j,l}\circ p_D)G_r$, $j=1,2$ and the data equation \eqref{4.2.5} for $\boldsymbol{x}\in\partial B_R$ can be reformulated as
\begin{equation*}
	\hat{\boldsymbol{v}}_l(\varsigma)=(\boldsymbol{N}_{s_l}^{c_1}(p_B,p_D,\varphi_{1,l}))(\varsigma)+(\boldsymbol{T}_{s_l}^{c_2}(p_B,p_D,\varphi_{2,l}))(\varsigma).	
\end{equation*}
\subsection{Discretization}
We adopt the Nystr$\mathrm{\ddot o}$m method \cite{Dong2019,nystrom} to discrete the boundary integrals \eqref{4.3}-\eqref{4.4}. The kernel $\widetilde{D}$ of the parametrized normal derivative integral operator can be written in the form
\[
\widetilde{D}(\theta,\eta; s_l,c)=\widetilde{D}_1(\theta,\eta;s_l,c)\ln\big(4\sin^2\frac{\theta-\eta}{2}\big)+\widetilde{D}_2(\theta,\eta;s_l,c),
\]
where
	\begin{align*}
	\widetilde{D}_1(\theta,\eta;s_l,c)&=\frac{{\rm i}s_l}{2\pi c}n(\theta)\cdot[p_D(\theta)-p_D(\eta)]\frac{J_1({\rm i}\frac{s_l}{c}|p_D(\theta)-p_D(\eta)|)}{|p_D(\theta)-p_D(\eta)|},\\
	\widetilde{D}_2(\theta,\eta;s_l,c)&=\tilde{D}(\theta,\eta; s_l,c)-\tilde{D}_1(\theta,\eta;s_l,c)\ln\big(4\sin^2\frac{\theta-\eta}{2}\big).	
	\end{align*}
It can be shown that the diagonal terms are
\[
\widetilde{D}_1(\theta,\theta;s_l,c)=0,\quad\widetilde{D}_2(\theta,\theta;s_l,c)=\frac{1}{2\pi}\frac{n(\theta)\cdot p_D''(\theta)}{|p_D'(\theta)|^2}.
\]
	
Similar to \cite{Dong2019},	we split the kernel $\tilde{H}$ of the parametrized tangential derivative integral operator into
\begin{equation*}
	\widetilde{H}(\theta,\eta;s_l,c)=\widetilde{H}_1(\theta,\eta;s_l,c)\frac{1}{\sin(\eta-\theta)}+\widetilde{H}_2(\theta,\eta;s_l,c)\ln\big(4\sin^2\frac{\theta-\eta}{2}\big)+\widetilde{H}_3(\theta,\eta;s_l,c),	
\end{equation*}
	where
	\begin{equation*}
		\begin{aligned}
			&\widetilde{H}_1(\theta,\eta;s_l,c)=\frac{1}{\pi}n(\theta)^\perp\cdot[p_D(\eta)-p_D(\theta)]\frac{\sin(\eta-\theta)}{|p_D(\theta)-p_D(\eta)|^2},\\
			&\widetilde{H}_2(\theta,\eta;s_l,c)=\frac{{\rm i}s_l}{2\pi c}n(\theta)^\perp\cdot[p_D(\theta)-p_D(\eta)]\frac{J_1({\rm i}\frac{s_l}{c}|p_D(\theta)-p_D(\eta)|)}{|p_D(\theta)-p_D(\eta)|},\\
			&\widetilde{H}_3(\theta,\eta;s_l,c)=\tilde{H}(\theta,\eta;s_l,c)-\tilde{H}_1(\theta,\eta;s_l,c)\frac{1}{\sin(\eta-\theta)}-\tilde{H}_2(\theta,\eta;s_l,c)\ln\big(4\sin^2\frac{\theta-\eta}{2}\big)
		\end{aligned}
	\end{equation*}
with diagonal entries given by
\[
\widetilde{H}_1(\theta,\theta;s_l,c)=\frac{1}{\pi},\quad \widetilde{H}_2(\theta,\theta;s_l,c)=0,\quad \tilde{H}_3(\theta,\theta;s_l,c)=0.
\]
	
Let $\eta_j^{(\tilde{n})}:=\pi j/\tilde{n}$, $j=0,\ldots,2\tilde{n}-1$, $\varsigma_i^{(\overline{n})}:=\pi i/\overline{n}$, $i=0,\ldots,2\overline{n}-1$ be equidistant sets of quadrature points. Then we obtain the fully discrete linear system:
\begin{equation*}
\begin{aligned}
	\omega_{1,l,i}^{(\tilde{n})}&=-\varphi_{1,l,i}^{(\tilde{n})}+\sum_{j=0}^{2\tilde{n}-1} X_{i,j;l,1} \varphi_{1,l,j}^{(\tilde{n})} +\sum_{j=0}^{2\tilde{n}-1} Y_{i,j;l,2} \varphi_{2,l,j}^{(\tilde{n})},\\
	\omega_{2,l,i}^{(\tilde{n})}&=\sum_{j=0}^{2\tilde{n}-1}Y_{i,j;l,1}\varphi_{1,l,j}^{(\tilde{n})} +\varphi_{2,l,i}^{(\tilde{n})}-\sum_{j=0}^{2\tilde{n}-1}X_{i,j;l,2} \varphi_{2,l,j}^{(\tilde{n})},
\end{aligned}
\end{equation*}
where $\omega_{1,l,i}^{(\tilde{n})}=\omega_{1,l}(\eta_i^{(\tilde{n})})$, $\omega_{2,l,i}^{(\tilde{n})}=\omega_{2,l}(\eta_i^{(\tilde{n})})$, $\varphi_{1,l,j}^{(\tilde{n})}=\varphi_{1,l}(\eta_{j}^{(\tilde{n})})$,
$\varphi_{2,l,j}^{(\tilde{n})}=\varphi_{2,l}(\eta_{j}^{(\tilde{n})})$ for $i,j=0,\ldots,2\tilde{n}-1$, and
\begin{equation*}
	\begin{aligned}
	X_{i,j;l,k}=& R_{|i-j|}^{(\tilde{n})}\tilde{D}_1(\eta_{i}^{(\tilde{n})},\eta_{j}^{(\tilde{n})};s_l,c_k)+\frac{\pi}{\tilde{n}}\tilde{D}_2(\eta_{i}^{(\tilde{n})},\eta_j^{(\tilde{n})};s_l,c_k)\\
	Y_{i,j;l,k}=& -T_{i-j}^{(\tilde{n})}\tilde{H}_1(\eta_{i}^{(\tilde{n})},\eta_j^{(\tilde{n})};s_l,c_k)+R_{|i-j|}^{(\tilde{n})}\tilde{H}_2(\eta_{i}^{(\tilde{n})},\eta_j^{(\tilde{n})};s_l,c_k)+\frac{\pi}{\tilde{n}}\tilde{H}_3(\eta_{i}^{(\tilde{n})},\eta_j^{(\tilde{n})};s_l,c_k)\\
	R_j^{(\tilde{n})}=&R_j^{(\tilde{n})}(0)=-\frac{2\pi}{\tilde{n}}\sum_{m=1}^{\tilde{n}-1}\frac{1}{m}\cos\frac{mj\pi}{\tilde{n}}-\frac{(-1)^j\pi}{\tilde{n}^2},\\
	T_j^{(\tilde{n})}=&T_j^{(\tilde{n})}(0)=\frac{2\pi}{\tilde{n}}\sum_{m=0}^{\tilde{m}}\sin\frac{(2m+1)j\pi}{\tilde{n}},\quad
	\tilde{m}=\begin{cases}
		(\tilde{n}-3)/2,&\tilde{n}=1,3,5,\ldots,\\
		\tilde{n}/2-1,&\tilde{n}=2,4,6,\ldots.
	\end{cases}
		\end{aligned}
\end{equation*}

\section{Reconstruction methods}
In this section, we introduce a system of nonlinear equations and develop corresponding reconstruction method for Problem 1.
\subsection{Nonlinear integral equation}
Based on the convolution quadrature method, we are going to solve a system of $s$-domain nonlinear integral equations for the inverse problem instead of solving the original time domain problem. Specifically, combining \eqref{field equation} and \eqref{3.12}, we obtain a system of field equations and data equation in $s$-domain		

%
	\begin{equation}\label{5.2}
		-\hat{g}_{1,l}+D_{s_l}^{c_1}\hat{g}_{1,l}+H_{s_l}^{c_2}\hat{g}_{2,l}=2\hat{f}_{1,l},
	\end{equation}
	\begin{equation}\label{5.3}
		H_{s_l}^{c_1}\hat{g}_{1,l}+\hat{g}_{2,l}-D_{s_l}^{c_2}\hat{g}_{2,l}=2\hat{f}_{2,l},
	\end{equation}
	\begin{equation}\label{5.1}
	\begin{aligned}
		(\boldsymbol{N}_{s_l}^{c_1}\hat{g}_{1,l})(\boldsymbol{x})+(\boldsymbol{T}_{s_l}^{c_2}\hat{g}_{2,l})(\boldsymbol{x})=\hat{\boldsymbol{v}}_l(\boldsymbol{x}),
	\end{aligned}
\end{equation}
for $l=0,\ldots,N$. The field equations and data equation \eqref{5.2}-\eqref{5.1} can be reformulated as parametrized integrals equations
	\begin{equation}\label{5.4}
		-\varphi_{1,l}+(D_{s_l}^{c_1}(p_D,\varphi_{1,l}))G_r+(H_{s_l}^{c_2}(p_D,\varphi_{2,l}))G_r=\omega_{1,l},
	\end{equation}
	\begin{equation}\label{5.5}
		\varphi_{2,l}+(H_{s_l}^{c_1}(p_D,\varphi_{1,l}))G_r-(D_{s_l}^{c_2}(p_D,\varphi_{2,l}))G_r=\omega_{2,l},
	\end{equation}
	\begin{equation}\label{5.6}
		\boldsymbol{N}_{s_l}^{c_1}(p_B,p_D,\varphi_{1,l})+\boldsymbol{T}_{s_l}^{c_2}(p_B,p_D,\varphi_{2,l})=\hat{\boldsymbol{v}}_l,
	\end{equation}
where $\omega_{j,l}=2(\hat{f}_{j,l}\circ p_D)G_r$, $j=1,2$.

We now seek a sequence of approximations to $\partial D$ by solving the field equations \eqref{5.4}-\eqref{5.5} and the data equation \eqref{5.6} in an alternating manner. Given an approximation for the boundary $\partial D$ one can solve \eqref{5.4}-\eqref{5.5} for $\varphi_{1,l}$ and $\varphi_{2,l}$. Then keeping $\varphi_{1,l}$ and $\varphi_{2,l}$ fixed, the update of the boundary $\partial D$ can be obtained by solving the linearized data equation \eqref{5.6} with respect to  $\partial D$.
\subsection{Iterative scheme}
The linearization of \eqref{5.6} with respect to $p_D$ requires the Fr$\mathrm{\acute{e}}$chet derivative of the parametrized integral operators $\boldsymbol{N}_s^c$ and $\boldsymbol{T}_s^c$, which can be explicitly calculated as following:
	\begin{equation}\label{5.7}
		\begin{aligned}	
			&\left( {\boldsymbol{N}_{s_l}^c}'[p_B,p_D,\varphi_{1,l}]q\right) (\varsigma)\\
			=&\int_{0}^{2\pi}\biggl\{ 
			-\frac{s_l}{4c}\left[\frac{{\rm i}s_l}{c} H_0^{(1)}({\rm i}\frac{s_l}{c}|p_B(\varsigma)-p_D(\eta)|)-\frac{H_1^{(1)}({\rm i}\frac{s_l}{c}|p_B(\varsigma)-p_D(\eta)|)}{|p_B(\varsigma)-p_D(\eta)|}\right] \frac{(p_B(\varsigma)-p_D(\eta))\cdot q(\eta)}{|p_B(\varsigma)-p_D(\eta)|^2}(p_B(\varsigma)-p_D(\eta))\\
			&-\frac{s_l}{4c}H_1^{(1)}({\rm i}\frac{s_l}{c}|p_B(\varsigma)-p_D(\eta)|)\frac{q(\eta)}{|p_B(\varsigma)-p_D(\eta)|}\\
			&+\frac{s_l}{4c}H_1^{(1)}({\rm i}\frac{s_l}{c}|p_B(\varsigma)-p_D(\eta)|)\frac{(p_B(\varsigma)-p_D(\eta))\cdot q(\eta)}{|p_B(\varsigma)-p_D(\eta)|^3}(p_B(\varsigma)-p_D(\eta))\biggr\}\varphi_{1,l}(\eta)\mathrm{d}\eta\\
			=&\int_{0}^{2\pi}\biggl\{\left[ -\frac{{\rm i}s_l^2}{4c^2}\frac{ H_0^{(1)}({\rm i}\frac{s_l}{c}|p_B(\varsigma)-p_D(\eta)|)}{|p_B(\varsigma)-p_D(\eta)|^2}+\frac{s_l}{2c}\frac{H_1^{(1)}({\rm i}\frac{s_l}{c}|p_B(\varsigma)-p_D(\eta)|)}{|p_B(\varsigma)-p_D(\eta)|^3}\right]\begin{pmatrix}
				b_1+R\cos\varsigma-p_1-r(\eta)\cos\eta)\\
				b_2+R\sin\varsigma-p_2-r(\eta)\sin\eta
			\end{pmatrix} \\
			&\qquad\times\bigg[(b_1+R\cos\varsigma-p_1-r(\eta)\cos\eta)(\Delta p_1+\Delta r(\eta)\cos\eta)\\
			&\qquad\quad+(b_2+R\sin\varsigma-p_2-r(\eta)\sin\eta)(\Delta p_2+\Delta r(\eta)\sin\eta)\bigg]\\
			&-\frac{s_l}{4c}\frac{H_1^{(1)}({\rm i}\frac{s_l}{c}|p_B(\varsigma)-p_D(\eta)|)}{|p_B(\varsigma)-p_D(\eta)|}\begin{pmatrix}
				\Delta p_1+\Delta r(\eta)\cos\eta\\
				\Delta p_2+\Delta r(\eta)\sin\eta
			\end{pmatrix}
			\biggr\}\varphi_{1,l}(\eta)\mathrm{d}\eta
		\end{aligned}
	\end{equation}
	and
	\begin{equation}\label{5.8}
		\begin{aligned}
			&\left( {\boldsymbol{T}_{s_l}^c}'[p_B,p_D,\varphi_{2,l}]q\right) (\varsigma)\\
			=&\int_{0}^{2\pi}\biggl\{ 
			-\frac{s_l}{4c}\left[\frac{{\rm i}s_l}{c} H_0^{(1)}({\rm i}\frac{s_l}{c}|p_B(\varsigma)-p_D(\eta)|)-\frac{H_1^{(1)}({\rm i}\frac{s_l}{c}|p_B(\varsigma)-p_D(\eta)|)}{|p_B(\varsigma)-p_D(\eta)|}\right]\\
			&\qquad\times \frac{(p_B(\varsigma)-p_D(\eta))\cdot q(\eta)}{|p_B(\varsigma)-p_D(\eta)|^2}\begin{pmatrix}
				P_{B,2}(\varsigma)-P_{D,2}(\eta)\\
				P_{D,1}(\eta)-P_{B,1}(\varsigma)
			\end{pmatrix}\\
			&+\frac{s_l}{4c}\frac{H_1^{(1)}({\rm i}\frac{s_l}{c}|p_B(\varsigma)-p_D(\eta)|)}{|p_B(\varsigma)-p_D(\eta)|}\begin{pmatrix}
				-q_2(\eta)\\
				q_1(\eta)
			\end{pmatrix}\\
			&+\frac{s_l}{4c}H_1^{(1)}({\rm i}\frac{s_l}{c}|p_B(\varsigma)-p_D(\eta)|)\frac{(p_B(\varsigma)-p_D(\eta))\cdot q(\eta)}{|p_B(\varsigma)-p_D(\eta)|^3}\begin{pmatrix}
				P_{B,2}(\varsigma)-P_{D,2}(\eta)\\
				P_{D,1}(\eta)-P_{B,1}(\varsigma)
			\end{pmatrix}\biggr\}\varphi_{2,l}(\eta)\mathrm{d}\eta\\
			=&\int_{0}^{2\pi}\biggl\{\left[ -\frac{{\rm i}s_l^2}{4c^2}\frac{ H_0^{(1)}({\rm i}\frac{s_l}{c}|p_B(\varsigma)-p_D(\eta)|)}{|p_B(\varsigma)-p_D(\eta)|^2}+\frac{s_l}{2c}\frac{H_1^{(1)}({\rm i}\frac{s_l}{c}|p_B(\varsigma)-p_D(\eta)|)}{|p_B(\varsigma)-p_D(\eta)|^3}\right]\begin{pmatrix}
				b_2+R\sin\varsigma-p_2-r(\eta)\sin\eta\\
				p_1+r(\eta)\cos\eta-b_1-R\cos\varsigma
			\end{pmatrix} \\
			&\qquad\times\bigg[(b_1+R\cos\varsigma-p_1-r(\eta)\cos\eta)(\Delta p_1+\Delta r(\eta)\cos\eta)\\
			&\qquad\quad+(b_2+R\sin\varsigma-p_2-r(\eta)\sin\eta)(\Delta p_2+\Delta r(\eta)\sin\eta)\bigg]\\
			&+\frac{s_l}{4c}\frac{H_1^{(1)}({\rm i}\frac{s_l}{c}|p_B(\varsigma)-p_D(\eta)|)}{|p_B(\varsigma)-p_D(\eta)|}\begin{pmatrix}
				-\Delta p_2-\Delta r(\eta)\sin\eta\\
				\Delta p_1+\Delta r(\eta)\cos\eta
			\end{pmatrix}
			\biggr\}\varphi_{2,l}(\eta)\mathrm{d}\eta
		\end{aligned}
	\end{equation}
where 
\[
q(\eta)=(q_1(\eta),q_2(\eta))^\top=(\Delta p_1,\Delta p_2)^\top+\Delta r(\eta)(\cos\eta,\sin\eta)^\top
\]
denotes the update of the boundary $\partial D$. Then the linearization of \eqref{5.6} leads to
\begin{equation}\label{update equation} {\boldsymbol{N}_{s_l}^{c_1}}'[p_B,p_D,\varphi_{1,l}]q+ {\boldsymbol{T}_{s_l}^{c_2}}'[p_B,p_D,\varphi_{2,l}]q=\boldsymbol{\omega}_l,
\end{equation}
where 
\[
\boldsymbol{\omega}_l=\hat{\boldsymbol{v}}_l-\left( \boldsymbol{N}_{s_l}^{c_1}(p_B,p_D,\varphi_{1,l})+\boldsymbol{T}_{s_l}^{c_2}(p_B,p_D,\varphi_{2,l})\right) .
\]

As usual, a stopping criterion is necessary to terminate the iteration. For our iterative procedure, the following relative error estimator is used:
\begin{equation}\label{error}
	E_{ll}:=\frac{\left\|\hat{\boldsymbol{v}}_l-\left( \boldsymbol{N}_{s_l}^{c_1}(p_B,p_D^{(ll)},\varphi_{1,l})+\boldsymbol{T}_{s_l}^{c_2}(p_B,p_D^{(ll)},\varphi_{2,l})\right) \right\| _{L^2 }}{\left\| \hat{\boldsymbol{v}}_l\right\| _{L^2}}\le \epsilon,
\end{equation}
where $\epsilon$ is a user-specified small positive constant depending on the noise level, and $p_D^{(ll)}$ is the
$ll$th approximation of the boundary $\partial D$.

We are now in a position to present the iterative algorithm for the inverse problem in the following Table.
\newpage
	\begin{table}[h]
		\begin{tabular}{cp{.8\textwidth}}
			\toprule
			\multicolumn{2}{l}{{\bf Algorithm:}\quad Iterative procedure for inverse obstacle scattering problem} \\
			\midrule
			{\bf Step 1} & Emit an incident plane wave $\boldsymbol{u}^{\rm inc}(\boldsymbol{x},t)$ with fixed $\lambda$, $\mu$, a fixed incident direction $\boldsymbol{d}$, and then collect the corresponding noisy scattered-field data $\boldsymbol{v}(\boldsymbol{y},t)$ at observation curve for the scatterer $D$; \\
			{\bf Step 2} & Take the discrete Fourier transform for the incident wave $\boldsymbol{u}^{\rm inc}$ and the scattered-field data $\boldsymbol{v}$; \\
			{\bf Step 3} & Select an initial star-like curve $\Gamma^{(0)}$ for the boundary $\partial D$, the error tolerance $\epsilon$ and a constant $loop$. Set $ll=0$;\\
			{\bf Step 4} & For the curve $\Gamma^{(ll)}$, find the densities $\varphi_{1,l}$ and $\varphi_{2,l}$ from \eqref{5.4}-\eqref{5.5} at $l=[ll/loop]$;\\
			{\bf Step 5} & Evaluate the error $E_{ll}$ defined in \eqref{error};\\
			{\bf Step 6} & If $E_{ll}\geq\epsilon$, then solve \eqref{update equation} to obtain the updated approximation $\Gamma^{(ll+1)}=\Gamma^{(11)}+q$ and set $ll=ll+1$ and go to Step 4. Otherwise, the current approximation $\Gamma^{(ll)}$ is served as the final reconstruction of $\partial D$. \\
			\bottomrule
		\end{tabular}
	\end{table}
	Here, the symbol $[A]$ represents the maximum integer that does not exceed the real number $A$ and the aim of {\bf Step 4} is to iterate for each fixed $l$.
\subsection{Discretization}
We use the Nystr$\mathrm{\ddot o}$m method described in section 4 for the full discretization of \eqref{5.4}-\eqref{5.5}. Now we discuss the discretization of the linearized equation \eqref{update equation} and obtain the update by using least squares with Tikhonov regularization \cite{Tikhonov}. As a finite-dimensional space to approximate the radial function $r$ and its update $\Delta r$, we choose the space of trigonometric polynomials of the form
\begin{equation*}
	\Delta r(\eta)=\sum_{m=0}^{M}\alpha_m\cos m\eta+\sum_{m=1}^{M}\beta_m\sin m\eta,
\end{equation*}
where the integer $M>1$ denotes the truncation number. 

Combing \eqref{5.7}-\eqref{update equation}, we get the fully discrete linear system
\begin{equation}\label{discretized system}
	\boldsymbol{B}_1^{(l)}(\varsigma_i^{(\overline{n})})\Delta p_1+\boldsymbol{B}_2^{(l)}(\varsigma_i^{(\overline{n})})\Delta p_2+\sum_{m=0}^{M}\alpha_m\boldsymbol{B}_{3,m}^{(l),r}(\varsigma_i^{(\overline{n})})+\sum_{m=1}^{M}\beta_m\boldsymbol{B}_{4,m}^{(l),r}(\varsigma_i^{(\overline{n})})=\boldsymbol{\omega}_l(\varsigma_i^{(\overline{n})})
\end{equation}
to determine the real coefficients $\Delta p_1$, $\Delta p_2$, $\alpha_m$ and $\beta_m$, where
\begin{equation*}
	\begin{aligned}
	\boldsymbol{B}_k^{(l)}(\varsigma_i^{(\overline{n})})&=\frac{\pi}{\tilde{n}}\sum_{j=0}^{2\tilde{n}-1}\boldsymbol{L}_k^{(l)}(\varsigma_i^{(\overline{n})},\eta_j^{\tilde{n}},c_1,c_2,\varphi_{1,l},\varphi_{2,l}),\quad k=1,2,\\
	\boldsymbol{B}_{k,m}^{(l),r}(\varsigma_i^{(\overline{n})})&=\frac{\pi}{\tilde{n}}\sum_{j=0}^{2\tilde{n}-1}\boldsymbol{L}_k^{(l),m}(\varsigma_i^{(\overline{n})},\eta_j^{\tilde{n}},c_1,c_2,\varphi_{1,l},\varphi_{2,l}),\quad k=3,4.		
\end{aligned}
\end{equation*}
For the detailed representations of $\boldsymbol{L}_k^{(l)}$ and $\boldsymbol{L}_k^{(l),m}$ we refer to the appendix.

In general, $2M+3\ll4\overline{n}$, and due to the ill-posedness, the overdetermined system \eqref{discretized system} is solved via the Tikhonov regularization. Hence the linear system \eqref{discretized system} is reformulated by minimizing the following function:
\begin{equation}\label{minimization}
	\begin{aligned}
&\sum_{i=0}^{2\overline{n}-1}|	\boldsymbol{B}_1^{(l)}(\varsigma_i^{(\overline{n})})\Delta p_1+\boldsymbol{B}_2^{(l)}(\varsigma_i^{(\overline{n})})\Delta p_2+\sum_{m=0}^{M}\alpha_m\boldsymbol{B}_{3,m}^{(l),r}(\varsigma_i^{(\overline{n})})+\sum_{m=1}^{M}\beta_m\boldsymbol{B}_{4,m}^{(l),r}(\varsigma_i^{(\overline{n})})-\boldsymbol{\omega}_l(\varsigma_i^{(\overline{n})})|^2\\
		&\quad+\lambda_0\bigg(|\Delta p_1|^2+|\Delta p_2|^2+2\pi[\alpha_0^2+\frac{1}{2}\sum_{m=1}^{M}(1+m^2)(\alpha_m^2+\beta_m^2)]\bigg)		
	\end{aligned}
\end{equation}
with a positive regularization parameter $\lambda_0$ and $H^1$ penalty term. It is easy to show that the minimizer of \eqref{minimization} is the solution of the system
\begin{equation}\label{end}
	\big(\lambda_0\tilde{I}+\Re(\tilde{B}_l^*\tilde{B}_l)\big)\xi=\Re(\tilde{B}_l^*\tilde{\omega_l}),
\end{equation}
where
\begin{equation*}
	\tilde{B}_l=\Big(\boldsymbol{B}_1^{(l)},\boldsymbol{B}_2^{(l)},\boldsymbol{B}_{3,0}^{(l),r},\dots,\boldsymbol{B}_{3,M}^{(l),r},\boldsymbol{B}_{4,1}^{(l),r},\dots,\boldsymbol{B}_{4,M}^{(l),r}\Big)_{(4\overline{n})\times(2M+3)}
\end{equation*}
and
\[
\xi=(\Delta p_1,\Delta p_2,\alpha_0,\dots,\alpha_M,\beta_1,\dots,\beta_M)^\top,
\]
\[
\tilde{I}=\mathrm{diag}\left\lbrace 1,1,2\pi,\pi(1+1^2),\dots,\pi(1+M^2),\pi(1+1^2),\dots,\pi(1+M^2)\right\rbrace ,
\]
\[
\tilde{\omega_l}=(\boldsymbol{\omega}_l(\varsigma_0^{(\overline{n})})^\top,
,\dots,\boldsymbol{\omega}_l(\varsigma_{2\overline{n}-1}^{(\overline{n})})^\top)^\top.
\]
Thus, we obtain the new approximation
\[
p_D^{new}(\hat{\boldsymbol{x}})=(\boldsymbol{p}+\Delta \boldsymbol{p})+(r(\hat{\boldsymbol{x}})+\Delta r(\hat{\boldsymbol{x}}))\hat{\boldsymbol{x}}.
\]
\begin{table}[t]
	\caption{Parametrization of the exact boundary curves}
	\begin{tabular}{lll}
		\toprule
		Type           &Parametrization\\
		\midrule
		apple-shaped   & $p_D(\theta)=\displaystyle\frac{1+0.9\cos{\theta}+0.1\sin(2\theta)}{1+0.75\cos{\theta}}(\cos{\theta},\sin{\theta}), \quad \theta\in [0,2\pi]$ \\
		~\\
		peanut-shaped  &
		$p_D(t)=\sqrt{0.25\cos^2{\theta}+\sin^2{\theta}}(\cos{\theta},\sin{\theta}), \quad
		\theta\in[0,2\pi]$\\
		\bottomrule
	\end{tabular}
\end{table}

\section{Numerical experiments}

In this section, we present some numerical examples to demonstrate the feasibility of the proposed iterative reconstruction methods. In all the examples, a single compressional plane wave is used to illuminate the obstacle. The scattered field data are numerically generated at 60 points, i.e., $\overline{n}=30$. In order to avoid the ``inverse crime'', we adopt 100 quadrature nodes for the direct problem and 64 quadrature nodes for the inverse problem.

\begin{figure}[h]
	\centering
	\subfigure[Reconstruction with 0.1$\%$ noise]{\includegraphics[width=0.45\textwidth]{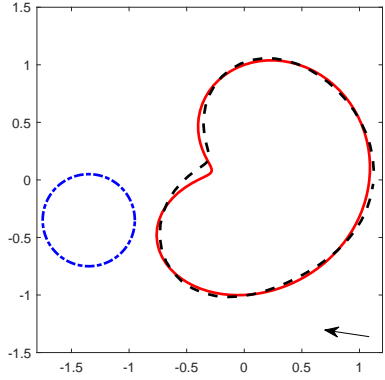}}
	\subfigure[Reconstruction with 1$\%$ noise]{\includegraphics[width=0.45\textwidth]{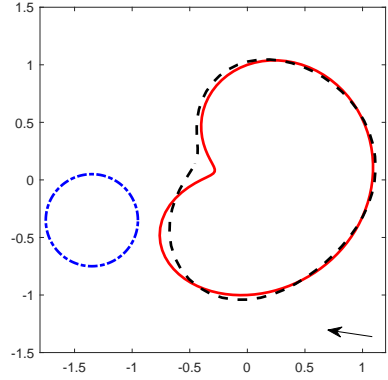}}
	\caption{Reconstructions of an apple-shaped obstacle at different levels of noise, the radius of the initial guess is $r_0=0.4$, the position of the initial guess is $(c_1^{(0)},c_2^{(0)})=(-1.35,-0.35)$ and the incident direction of the incident wave is $(\cos\theta,\sin\theta )$, $\theta=15\pi/16$.}\label{figure1}
\end{figure}
\begin{figure}[h]
	\centering
	\subfigure[Reconstruction with 0.1$\%$ noise]{\includegraphics[width=0.45\textwidth]{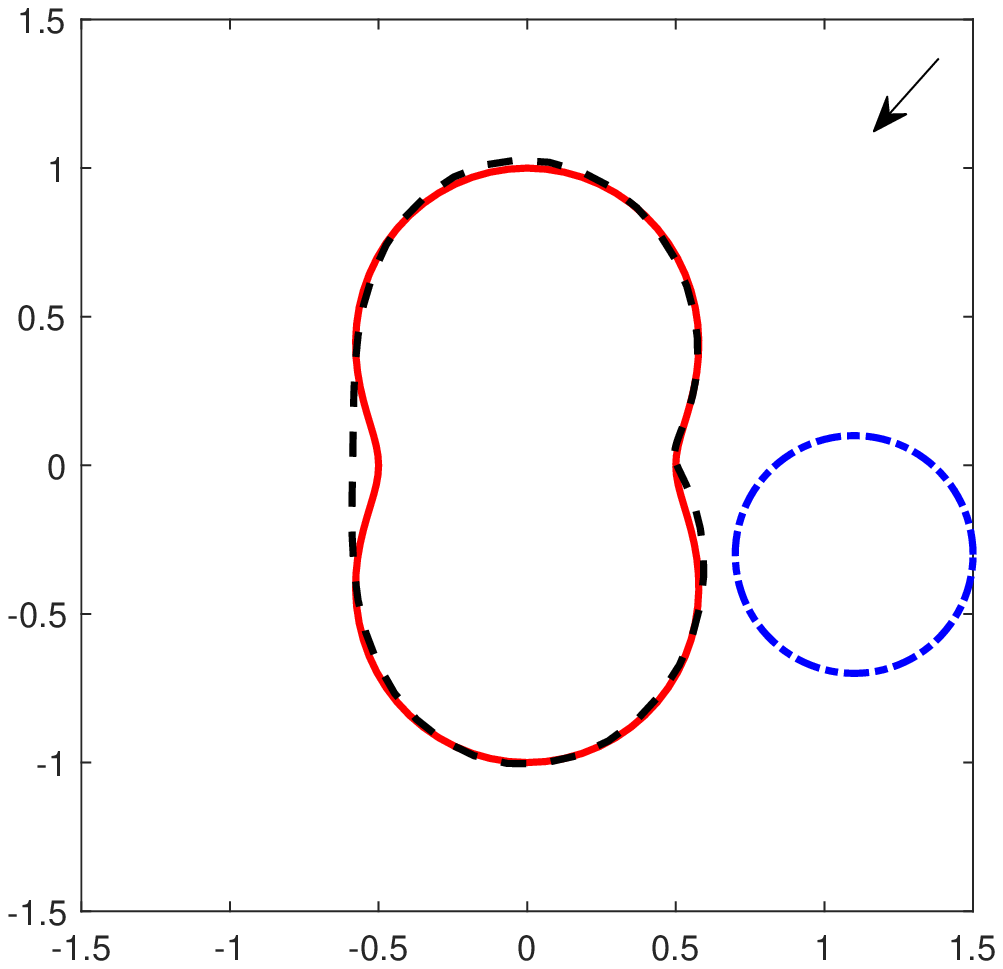}}
	\subfigure[Reconstruction with 1$\%$ noise]{\includegraphics[width=0.45\textwidth]{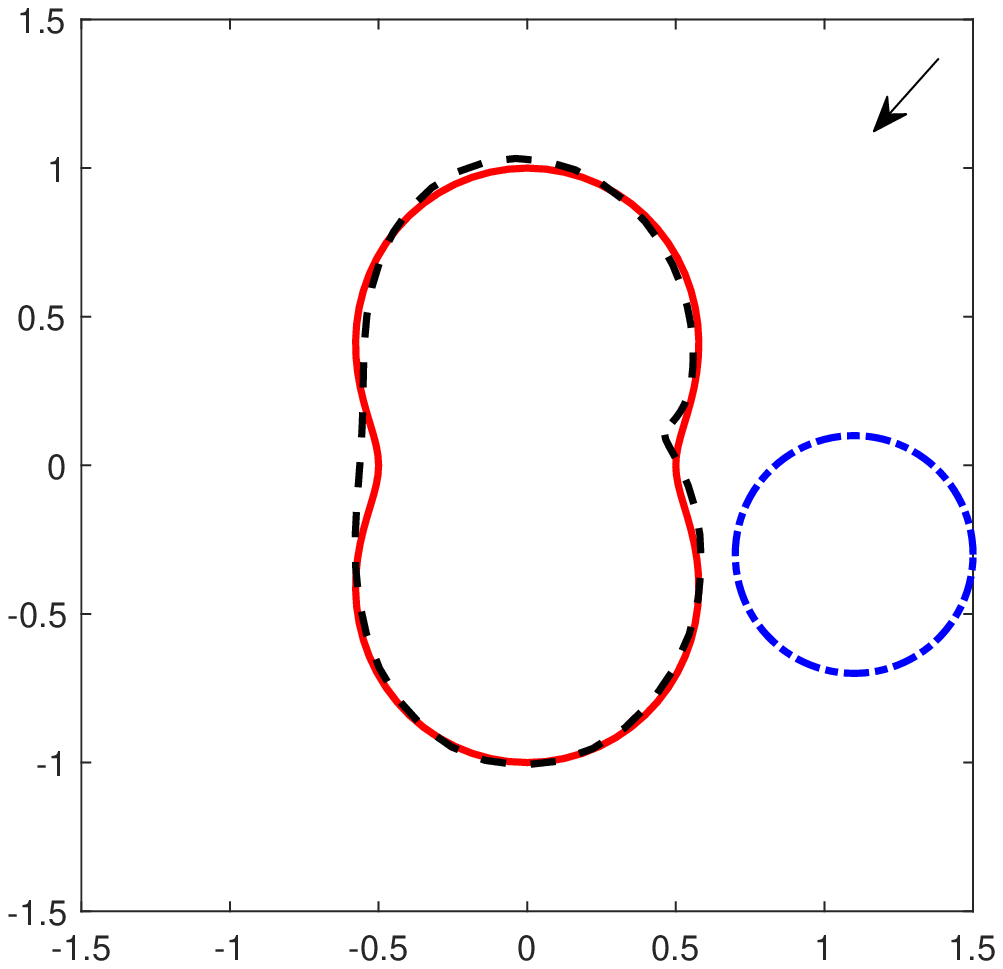}}
	\caption{Reconstructions of a peanut-shaped obstacle at different levels of noise, the radius of the initial guess is $r_0=0.4$, the position of the initial guess is $(c_1^{(0)},c_2^{(0)})=(1.1,-0.3)$ and the incident direction of the incident wave is $(\cos\theta,\sin\theta )$, $\theta=21\pi/16$.}\label{figure2}
\end{figure}
\begin{figure}[t]
	\centering
	\subfigure{\includegraphics[width=0.45\textwidth]{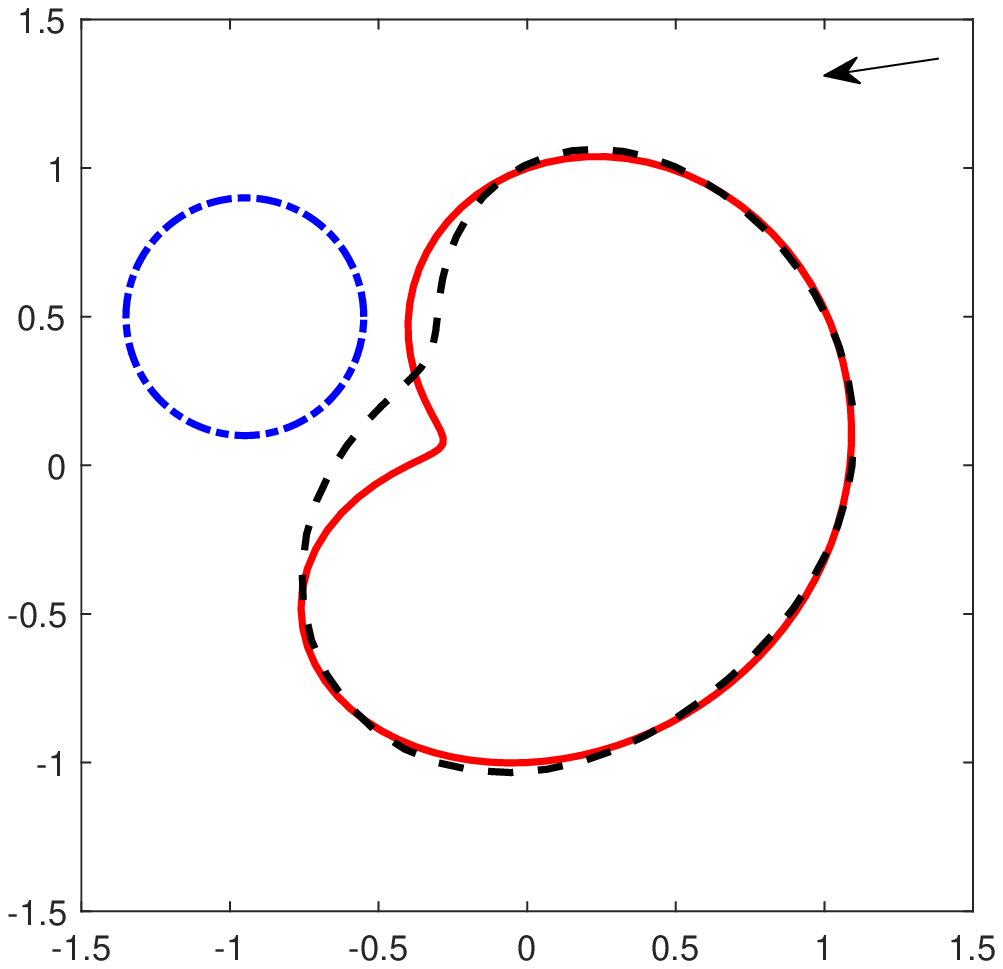}}
	\subfigure{\includegraphics[width=0.45\textwidth]{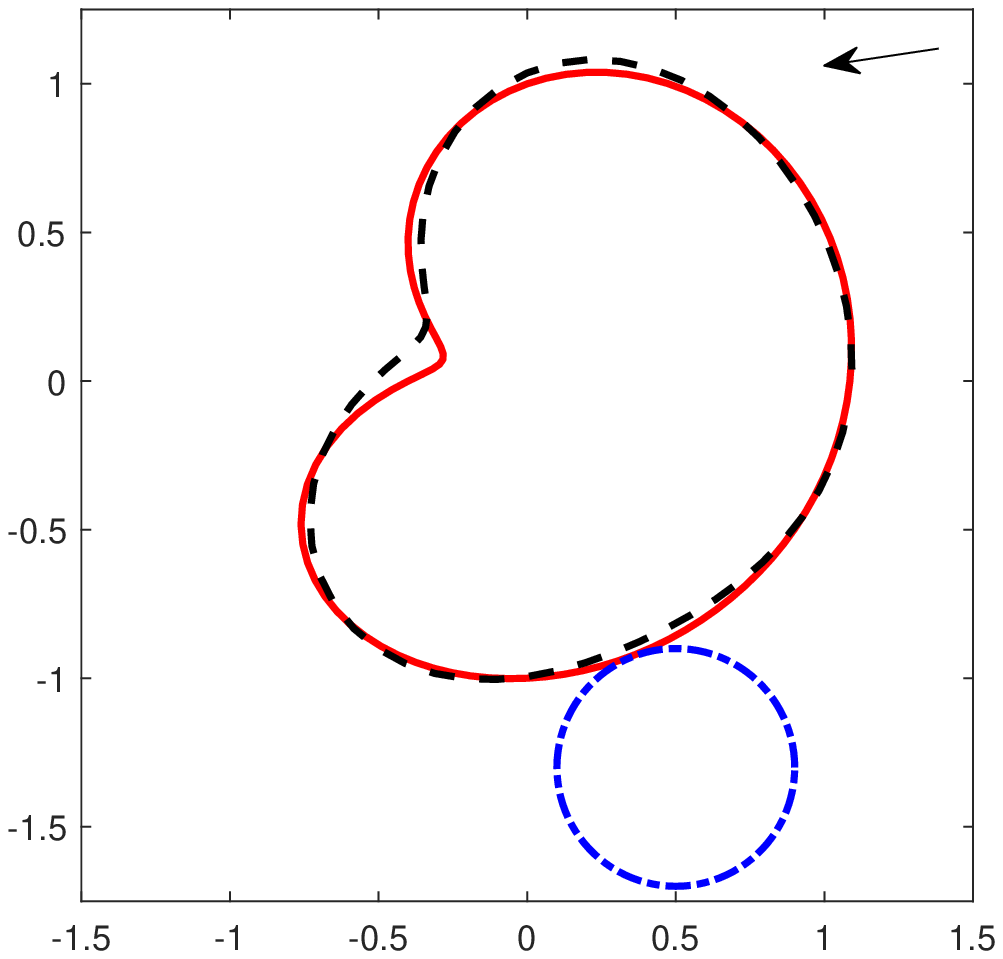}}
	\caption{Reconstructions of an apple-shaped obstacle at different positions of the initial guess $(c_1^{(0)},c_2^{(0)})=(-0.95,0.5)$$({\rm left})$ and $(c_1^{(0)},c_2^{(0)})=(0.5,-1.3)$$({\rm right})$, the radius of the initial guess is $r_0=0.4$, the noise level is 0.1$\%$ and the incident direction is $(\cos\theta,\sin\theta )$, $\theta=17\pi/16$.}\label{figure3}
\end{figure}
\begin{figure}[t]
	\centering
	\subfigure{\includegraphics[width=0.45\textwidth]{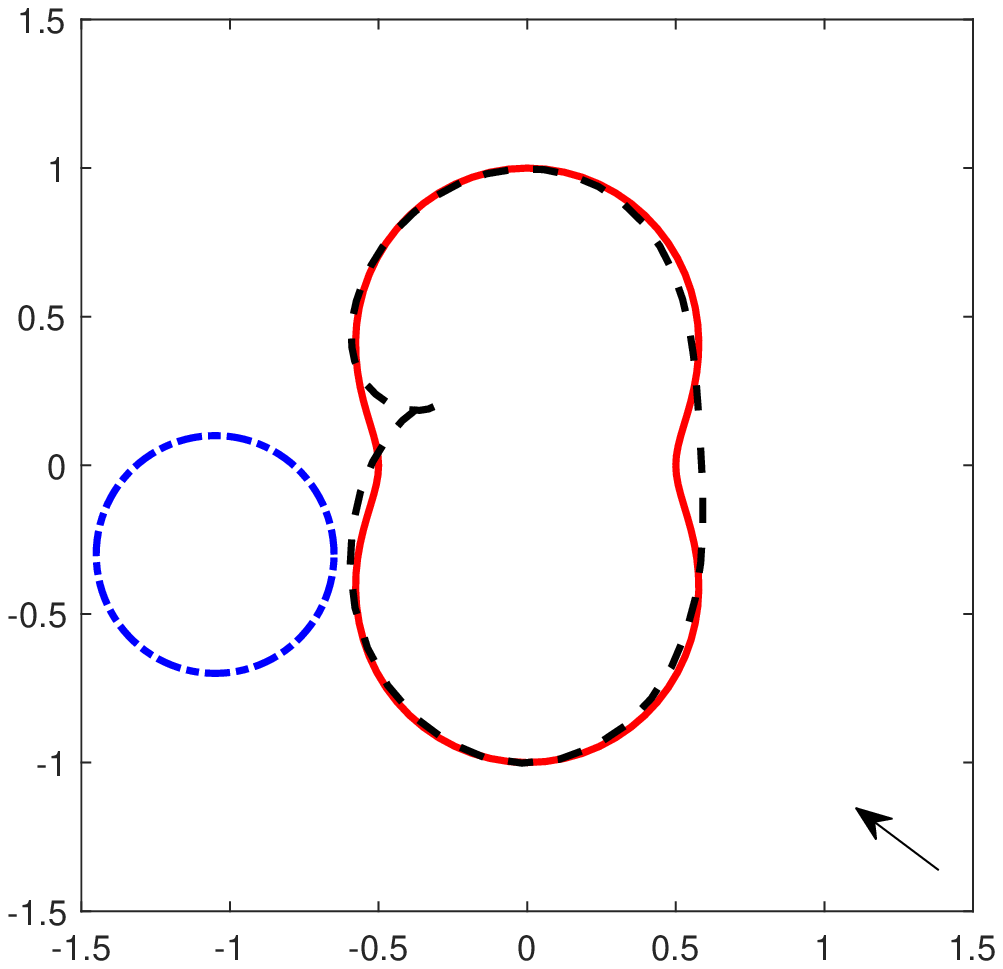}}
	\subfigure{\includegraphics[width=0.45\textwidth]{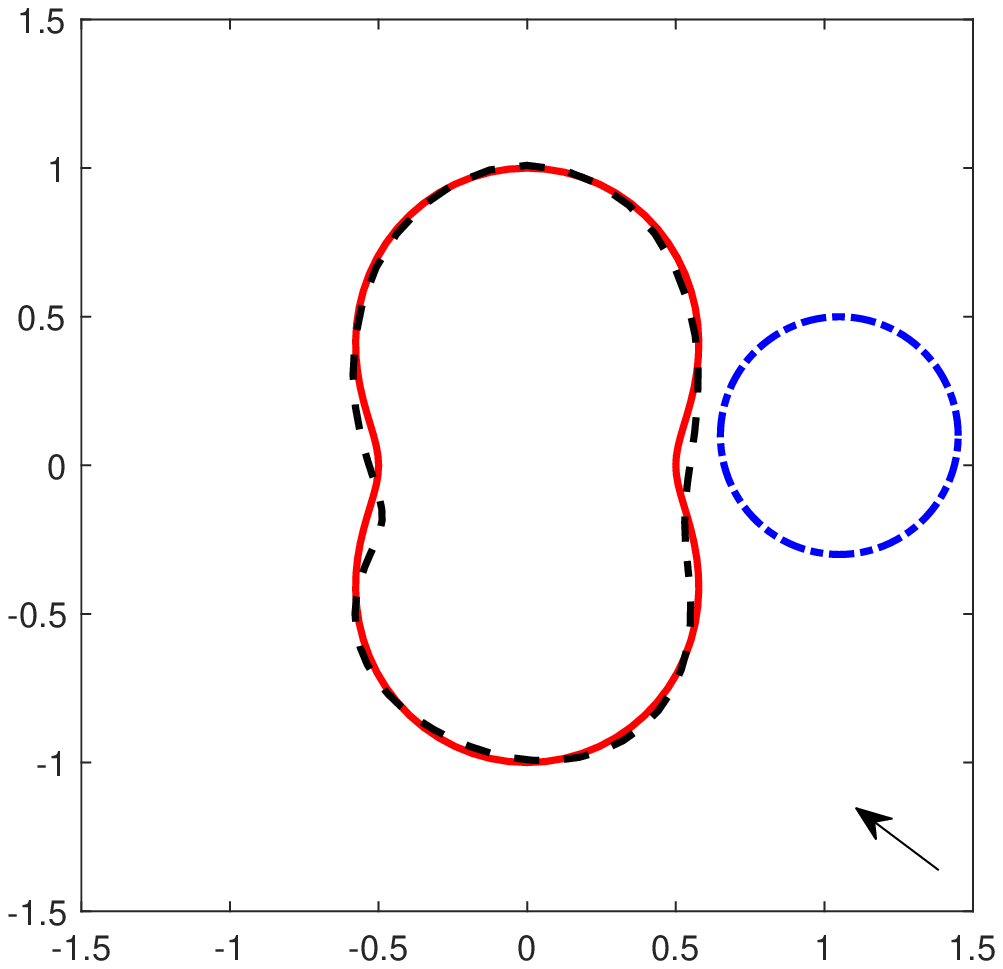}}
	\caption{Reconstructions of a peanut-shaped obstacle at different positions of the initial guess $(c_1^{(0)},c_2^{(0)})=(-1.05,-0.3)$$({\rm left})$ and $(c_1^{(0)},c_2^{(0)})=(1.05,0.1)$$({\rm right})$, the radius of the initial guess is $r_0=0.4$, the noise level is 0.1$\%$ and the incident direction is $(\cos\theta,\sin\theta )$, $\theta=3\pi/4$.}\label{figure4}
\end{figure}

\begin{figure}[t]
	\centering
	\subfigure{\includegraphics[width=0.45\textwidth]{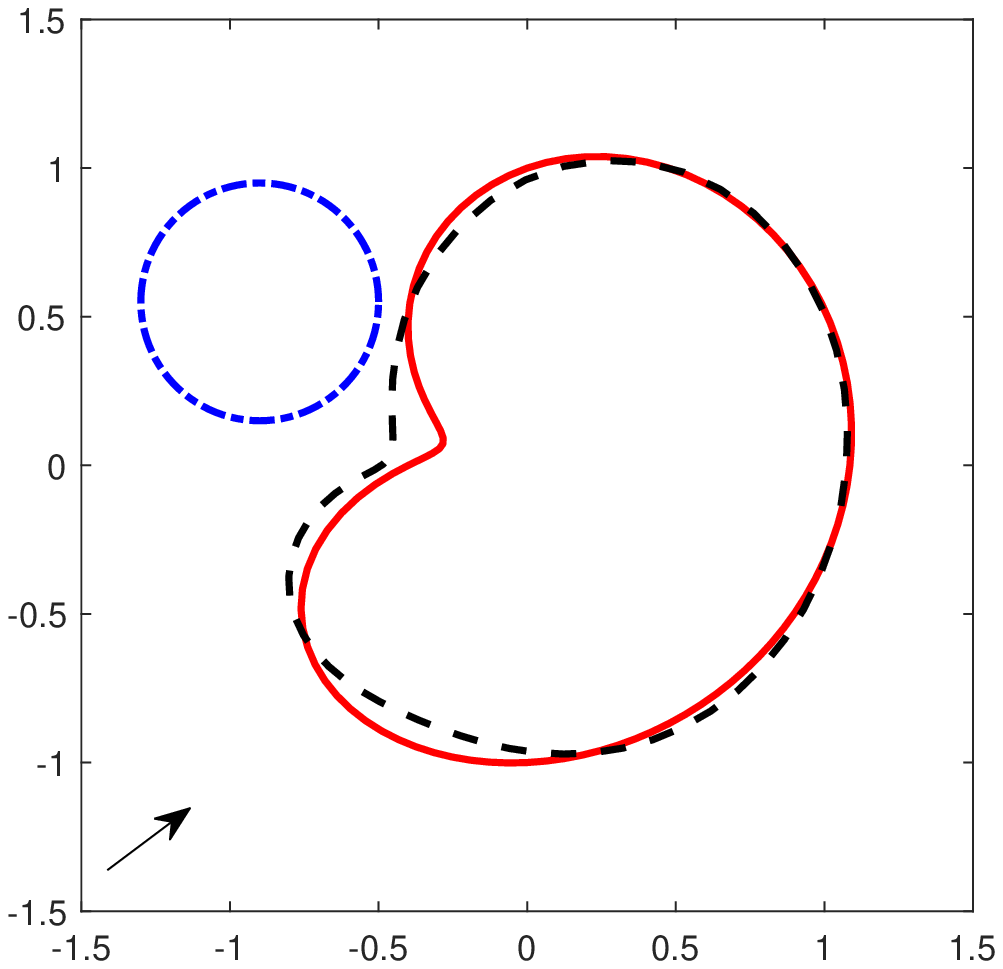}}
	\subfigure{\includegraphics[width=0.45\textwidth]{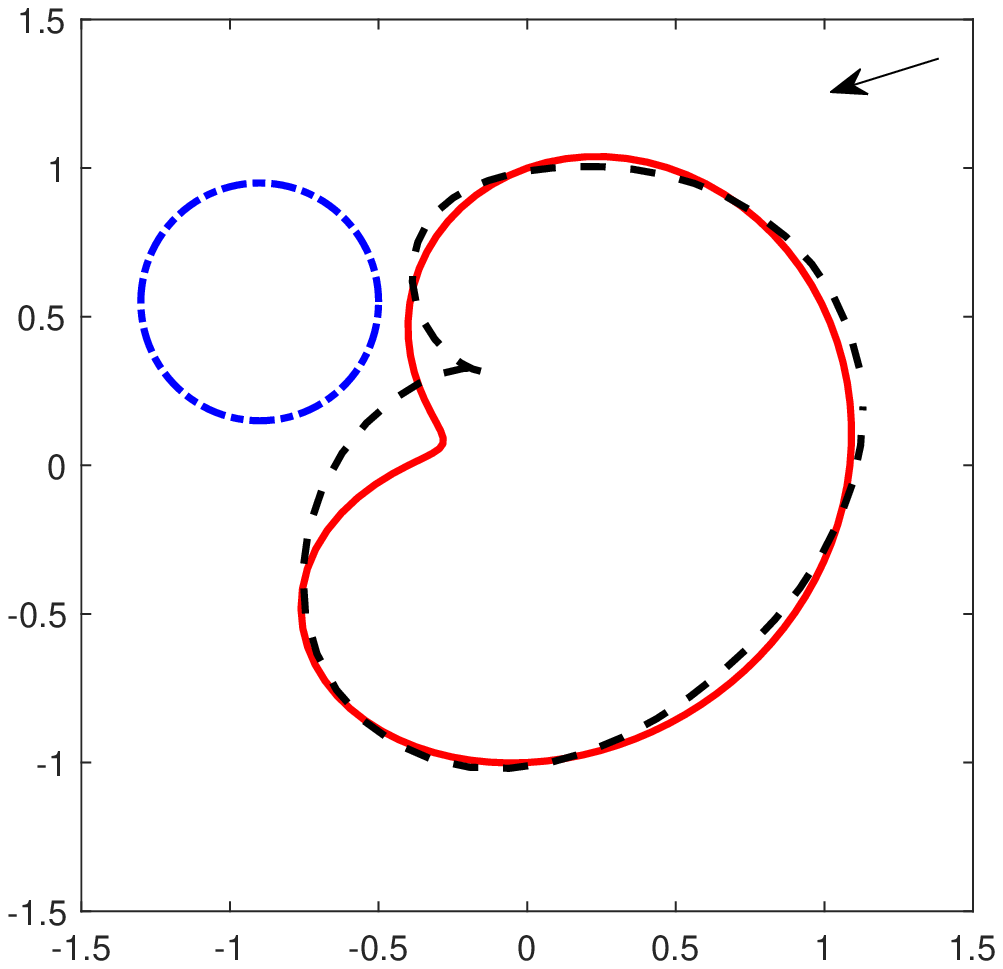}}
	\subfigure{\includegraphics[width=0.45\textwidth]{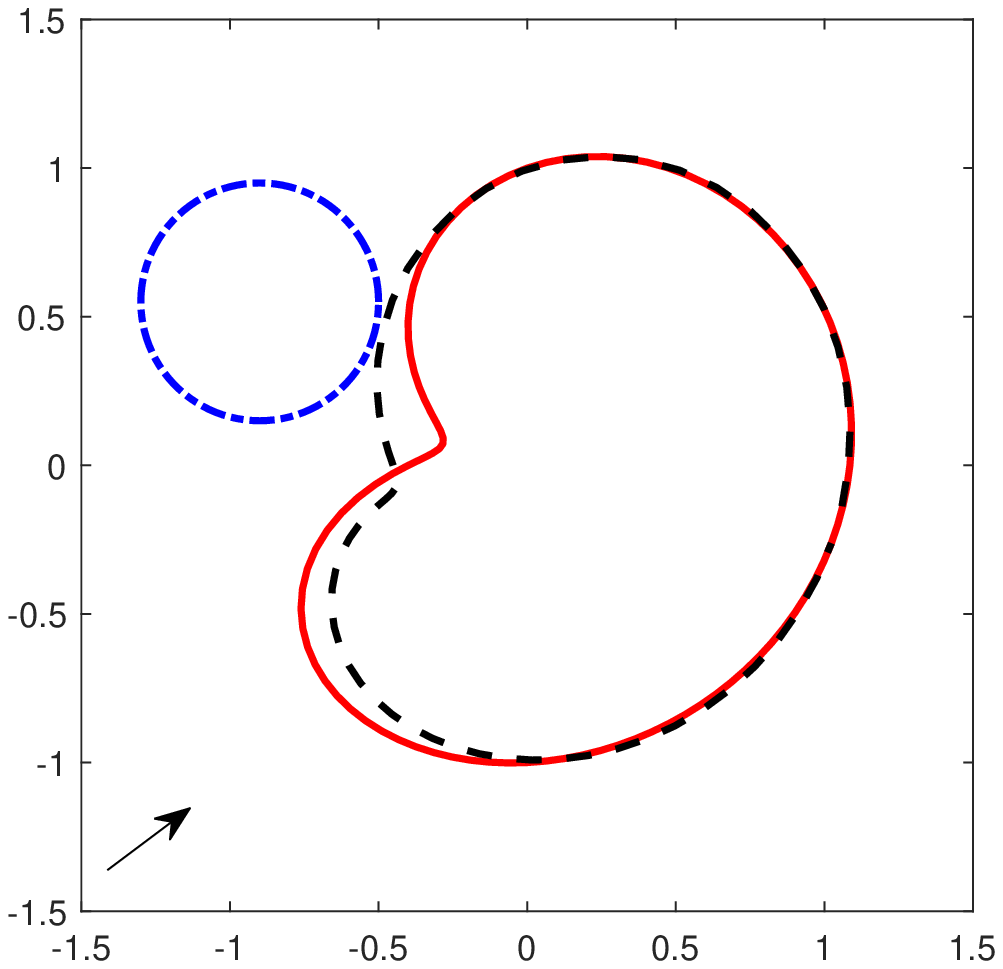}}
	\subfigure{\includegraphics[width=0.45\textwidth]{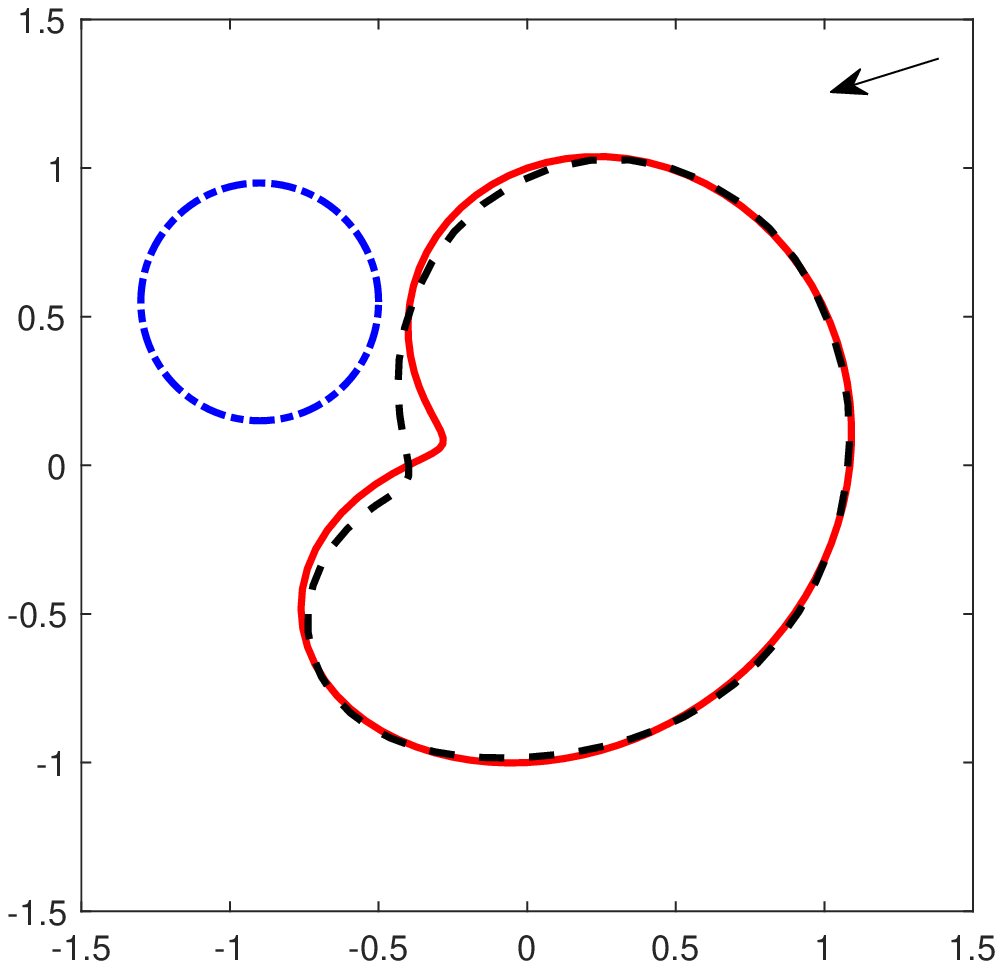}}
	\caption{Reconstructions of an apple-shaped obstacle with different incident directions $(\cos\theta,\sin\theta)$, $\theta=1\pi/4$$({\rm left})$, $\theta=9\pi/8$$({\rm right})$ or different radii of observation curve $R=2.5$$({\rm top})$, $R=2$$({\rm bottom})$. The radius of the initial guess is $r_0=0.4$, the position of the initial guess is $(c_1^{(0)},c_2^{(0)})=(-0.9,0.55)$ and the noise level is 0.1$\%$.}\label{figure5}
\end{figure}
\begin{figure}[t]
	\centering
	\subfigure{\includegraphics[width=0.45\textwidth]{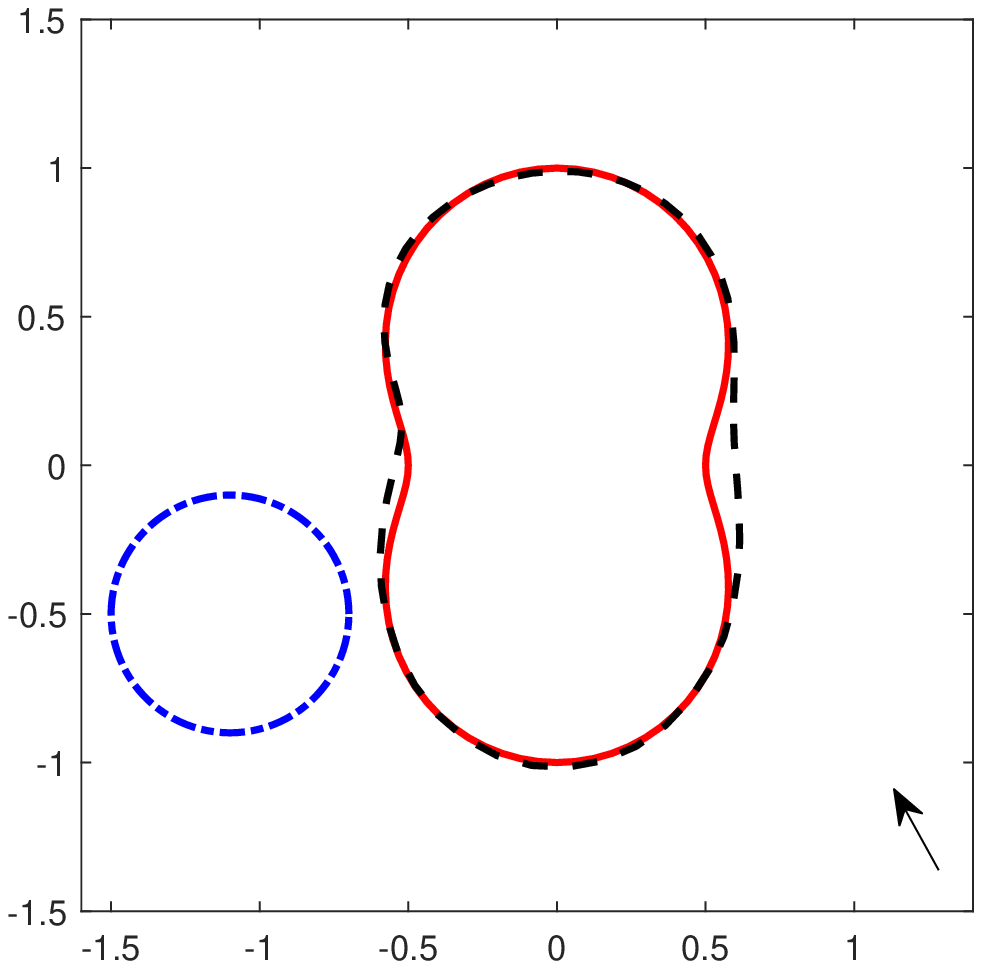}}
	\subfigure{\includegraphics[width=0.45\textwidth]{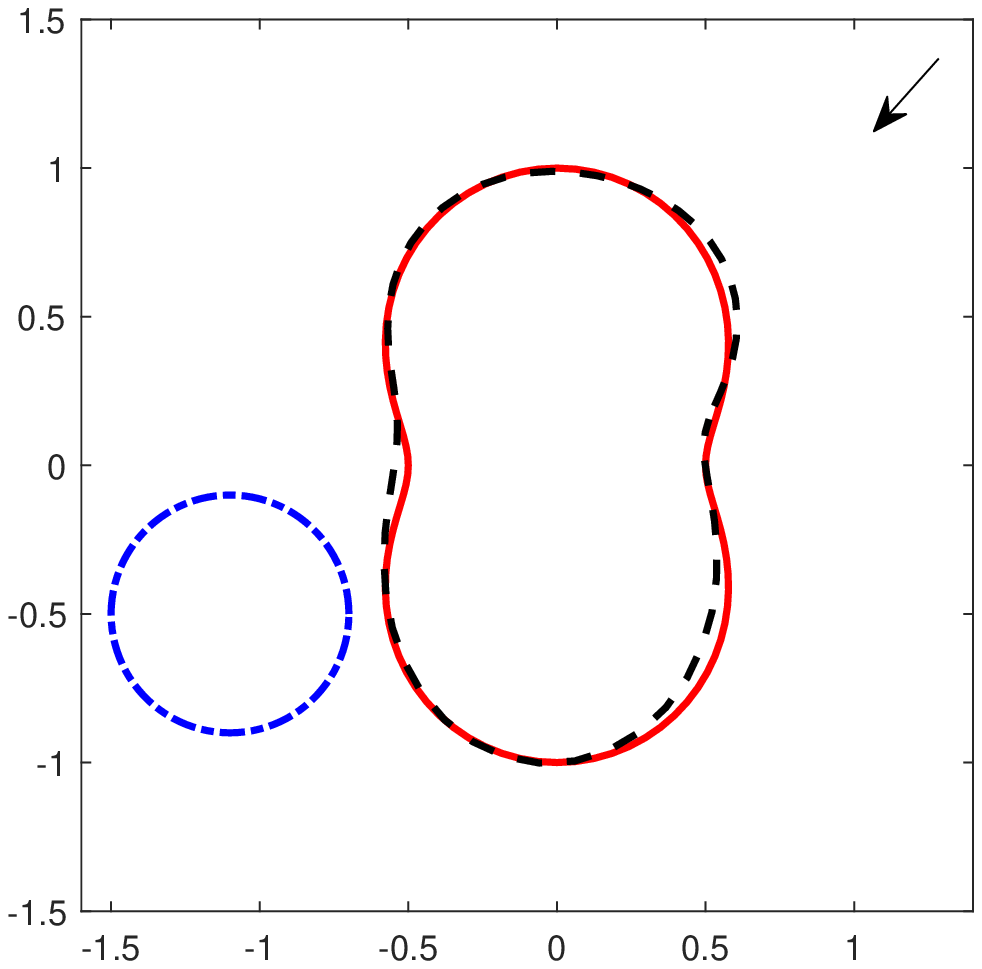}}
	\subfigure{\includegraphics[width=0.45\textwidth]{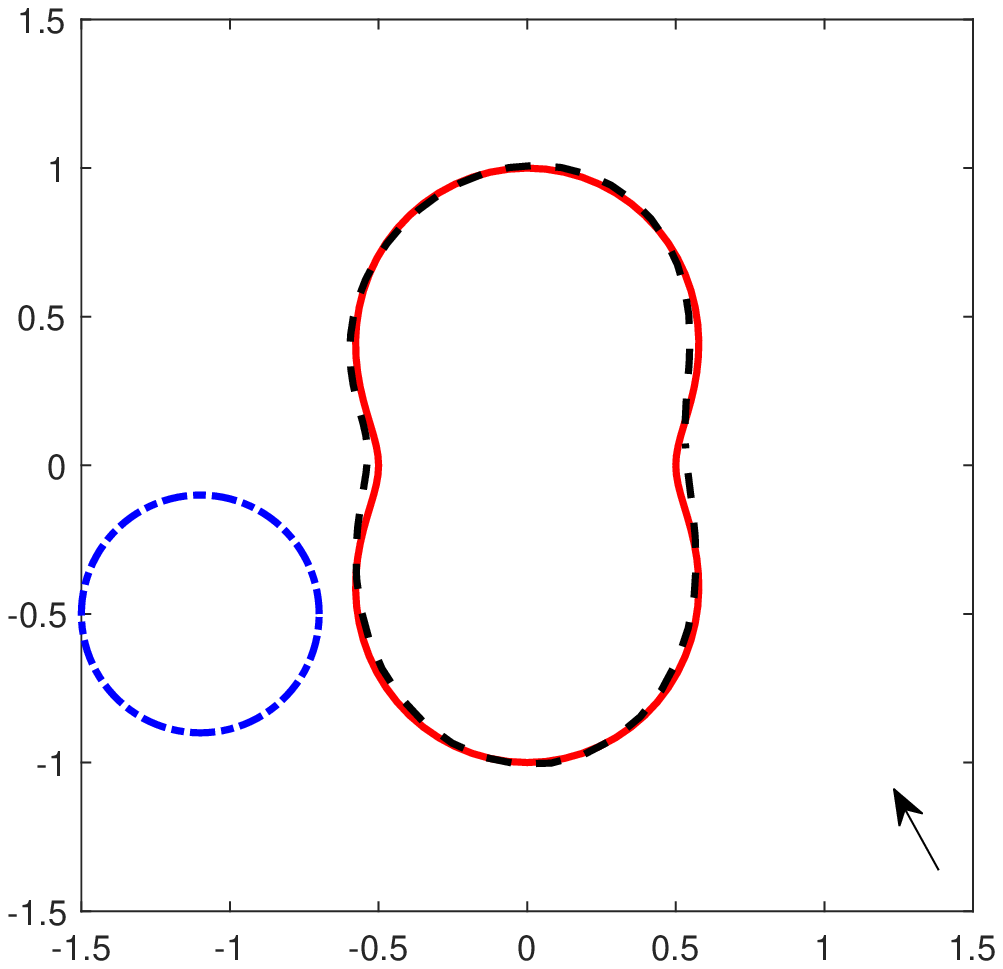}}
	\subfigure{\includegraphics[width=0.45\textwidth]{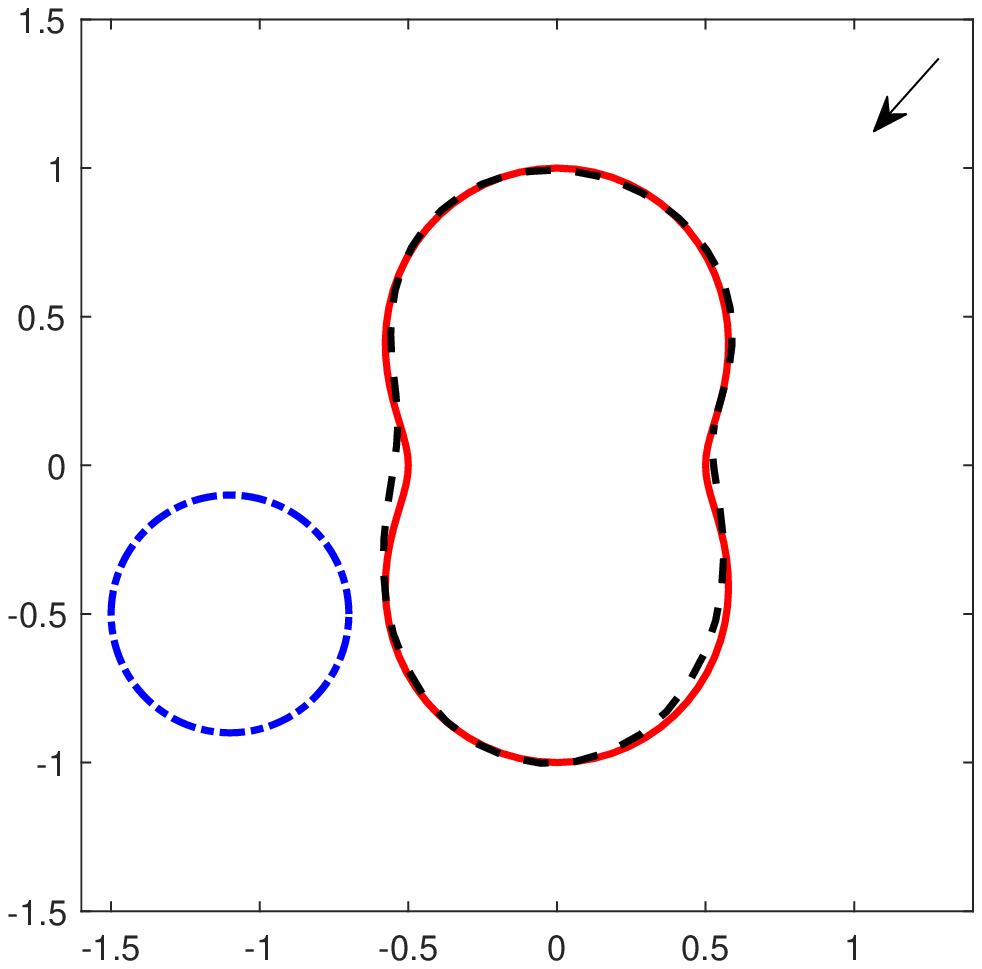}}
	\caption{Reconstructions of a peanut-shaped obstacle with different incident directions $(\cos\theta,\sin\theta)$, $\theta=5\pi/8$$({\rm left})$, $\theta=21\pi/16$$({\rm right})$ or different radii of observation curve $R=2.5$$({\rm top})$, $R=2$$({\rm bottom})$. The radius of the initial guess is $r_0=0.4$, the position of the initial guess is $(c_1^{(0)},c_2^{(0)})=(-1.1,-0.5)$ and the noise level is 0.1$\%$.}\label{figure6}
\end{figure}

To test stability, the noisy data $\boldsymbol{v}_\delta(\boldsymbol{y},t)$ is constructed in the following way:
\[
\boldsymbol{v}_\delta=\boldsymbol{v}(1+\delta\Theta),
\]
where $\Theta$ is normally distributed random numbers ranging in $[-1,1]$, $\delta>0$ is the relative noise level. In the iteration, we obtain the update $\xi$ from a scaled Newton step by using the Tikhonov regularization and $H^1$ penalty term, i.e.,
\[
\xi=\rho\bigg(\lambda_0\tilde{I}+\Re(\tilde{B}_l^*\tilde{B}_l)\bigg)^{-1}\Re(\tilde{B}_l^*\tilde{f_l}),
\]
where the scaling factor $\rho\ge 0$ is fixed throughout the iterations. Analogously to \cite{lambda}, the
regularization parameters $\lambda_0$ in \eqref{end} are chosen as
\[
\lambda_0^{(ll)}=\left\| \hat{\omega}_{l,\delta}^{\partial B_R}-AL_l(p^{(ll)},p_B,\hat{\phi}_l)\right\| _{L^2},\quad l=[ll/loop],~ll=0,1,\dots
\]

Analogously to Section 4.1 of \cite{CQ}, we adopt a strategy for the reduction of wavenumbers $s_l$. If the  $L^2$ norm of right hand $\omega_l(\theta)=[\omega_{1,l}(\theta);\omega_{2,l}(\theta)]=[-2\nu\cdot\hat{\boldsymbol{u}}_l^{\rm inc}(p_D(\theta))G_r;-2\tau\cdot\hat{\boldsymbol{u}}_l^{\rm inc}(p_D(\theta))G_r]$ is less than the tolerance: $\tilde{\epsilon}=10^{-6}$, the update $\xi=\boldsymbol{0}$. In all figures, the exact boundary curves are displayed by solid lines, the reconstructed boundary curves are depicted by dashed lines - -, and all the initial guesses are taken to be a circle with radius $r_0$ which is indicated by the dash-dotted lines · -.  In addition, we assume that $f(t)$ is a causal function and the incident wave is chosen in form of
\[
\boldsymbol{u}^{\rm inc}(\boldsymbol{x},t):=\boldsymbol{d}f(c_1t+\boldsymbol{x}\cdot\boldsymbol{d}-R_0),\quad R_0=1.2,\quad f(t):=\sin^3(3t)|_{t>0},
\]
where $\boldsymbol{d}$ denotes the incident direction marked by an arrow in the figures. In order to ensure that the energy of the scattered data inside the interested domain is negligible when $t>T$, the incident field is required to be zero when $t>5$. Throughout all the numerical examples, we take upper limit of time $T=10$, the scaling factor $\rho=0.9$, the cycle-index for per $l$, $loop=4$, and the truncation $M=3$. The Lam$\acute{\rm e}$ constants are $\lambda=3.88$, $\mu=2.56$. If there  is no special instruction, the radius of observation curve is $R=2$, and the center of observation curve is $(0,0)$. We presents the reconstruction results for two commonly used examples: an apple-shaped obstacle and a peanut-shaped obstacle. The parametrization of the exact boundary curves for these two obstacles are given in Table 1.


We investigate an inverse elastic scattering problem of reconstructing a rigid obstacle from the scattered-field data in the time domain, including the shape and the position. The reconstructions with 0.1$\%$ noise and 1$\%$ noise for the apple-shaped and peanut-shaped obstacles are shown in Figures 1 and 2, respectively. For the fixed incident direction, Figures 3 and 4 show the reconstructions of the apple-shaped and the peanut-shaped obstacles by using different initial guesses; For the same initial guess, Figures 5 and 6 show the reconstructions of the apple-shaped and the peanut-shaped obstacles by using different directions of the incident wave or different radii of observation curve. As shown in the numerical results, the shape and location of the obstacle can be satisfactorily reconstructed.
\section{Conclusions}
In this paper, we have studied the two-dimensional inverse elastic scattering problem by the time domain scattered field data for a single incident plane wave. Based on the Helmholtz decomposition, the initial-boundary value problem of the time domain Navier equation is converted into a coupled initial-boundary value problem of wave equations, and the uniqueness of the solution for this coupled problem is proved. We introduce the  single layer potential and establish coupled boundary integral equations, then we prove the uniqueness of the solution for the coupled boundary integral equations. The convolution quadrature method combined with the nonlinear integral equation method is developed for the inverse problem. Numerical examples are presented to demonstrate the effectiveness and stability of the proposed method. Future work includes the application to inverse acoustic-elastic interaction
scattering problems and other scattering models.
\newpage
\section*{Appendix}
In the Appendix, we give the detailed representation of  $\boldsymbol{L}_k^{(l)}$ and $\boldsymbol{L}_k^{(l),m}$, i.e.
\[
\boldsymbol{L}_1^{(l)}(\varsigma,\eta,c_1,c_2,\varphi_{1,l},\varphi_{2,l})=\begin{pmatrix}
	L_{1,1}^{(l)}(\varsigma,\eta,c_1,c_2,\varphi_{1,l},\varphi_{2,l})\\
	L_{1,2}^{(l)}(\varsigma,\eta,c_1,c_2,\varphi_{1,l},\varphi_{2,l})
\end{pmatrix},
\]
\[
\boldsymbol{L}_2^{(l)}(\varsigma,\eta,c_1,c_2,\varphi_{1,l},\varphi_{2,l})=\begin{pmatrix}
	L_{2,1}^{(l)}(\varsigma,\eta,c_1,c_2,\varphi_{1,l},\varphi_{2,l})\\
	L_{2,2}^{(l)}(\varsigma,\eta,c_1,c_2,\varphi_{1,l},\varphi_{2,l})
\end{pmatrix},
\]
\[\boldsymbol{L}_3^{(l),m}(\varsigma,\eta,c_1,c_2,\varphi_{1,l},\varphi_{2,l})=\begin{pmatrix}
	L_{3,1}^{(l),m}(\varsigma,\eta,c_1,c_2,\varphi_{1,l},\varphi_{2,l})\\
	L_{3,2}^{(l),m}(\varsigma,\eta,c_1,c_2,\varphi_{1,l},\varphi_{2,l})
\end{pmatrix},
\]
\[
\boldsymbol{L}_4^{(l),m}(\varsigma,\eta,c_1,c_2,\varphi_{1,l},\varphi_{2,l})=\begin{pmatrix}
	L_{4,1}^{(l),m}(\varsigma,\eta,c_1,c_2,\varphi_{1,l},\varphi_{2,l})\\
	L_{4,2}^{(l),m}(\varsigma,\eta,c_1,c_2,\varphi_{1,l},\varphi_{2,l})
\end{pmatrix},
\]
where
%
\begin{equation*}
	\begin{aligned}
		&L_{1,1}^{(l)}(\varsigma,\eta,c_1,c_2,\varphi_{1,l},\varphi_{2,l})\\
		=&\bigl\{\big( -\frac{{\rm i}s_l^2}{4c_1^2}\frac{ H_0^{(1)}({\rm i}\frac{s_l}{c_1}|p_B(\varsigma)-p_D(\eta)|)}{|p_B(\varsigma)-p_D(\eta)|^2}+\frac{s_l}{2c_1}\frac{H_1^{(1)}({\rm i}\frac{s_l}{c_1}|p_B(\varsigma)-p_D(\eta)|)}{|p_B(\varsigma)-p_D(\eta)|^3}\big)\big(b_1+R\cos\varsigma-p_1-r(\eta)\cos\eta\big)^2\varphi_{1,l}(\eta)\\
		&-\frac{s_l}{4c_1}\frac{H_1^{(1)}({\rm i}\frac{s_l}{c_1}|p_B(\varsigma)-p_D(\eta)|)}{|p_B(\varsigma)-p_D(\eta)|}\varphi_{1,l}(\eta)\\
		&+\big(-\frac{{\rm i}s_l^2}{4c_2^2}\frac{ H_0^{(1)}({\rm i}\frac{s_l}{c_2}|p_B(\varsigma)-p_D(\eta)|)}{|p_B(\varsigma)-p_D(\eta)|^2}+\frac{s_l}{2c_2}\frac{H_1^{(1)}({\rm i}\frac{s_l}{c_2}|p_B(\varsigma)-p_D(\eta)|)}{|p_B(\varsigma)-p_D(\eta)|^3}\big)\\
		&\times \big(b_2+R\sin\varsigma-p_2-r(\eta)\sin\eta\big)\big(b_1+R\cos\varsigma-p_1-r(\eta)\cos\eta\big)\varphi_{2,l}(\eta)
		\bigr\},
	\end{aligned}
\end{equation*}
\begin{equation*}
	\begin{aligned}
		&L_{2,1}^{(l)}(\varsigma,\eta,c_1,c_2,\varphi_{1,l},\varphi_{2,l})\\
		=&\bigl\{\big( -\frac{{\rm i}s_l^2}{4c_1^2}\frac{ H_0^{(1)}({\rm i}\frac{s_l}{c_1}|p_B(\varsigma)-p_D(\eta)|)}{|p_B(\varsigma)-p_D(\eta)|^2}+\frac{s_l}{2c_1}\frac{H_1^{(1)}({\rm i}\frac{s_l}{c_1}|p_B(\varsigma)-p_D(\eta)|)}{|p_B(\varsigma)-p_D(\eta)|^3}\big)\\
		&\times\big(b_2+R\sin\varsigma-p_2-r(\eta)\sin\eta\big)\big(b_1+R\cos\varsigma-p_1-r(\eta)\cos\eta\big)\varphi_{1,l}(\eta)\\
		&+\big(-\frac{{\rm i}s_l^2}{4c_2^2}\frac{ H_0^{(1)}({\rm i}\frac{s_l}{c_2}|p_B(\varsigma)-p_D(\eta)|)}{|p_B(\varsigma)-p_D(\eta)|^2}+\frac{s_l}{2c_2}\frac{H_1^{(1)}({\rm i}\frac{s_l}{c_2}|p_B(\varsigma)-p_D(\eta)|)}{|p_B(\varsigma)-p_D(\eta)|^3}\big) \big(b_2+R\sin\varsigma-p_2-r(\eta)\sin\eta\big)^2\varphi_{2,l}(\eta)\\
		&-\frac{s_l}{4c_2}\frac{H_1^{(1)}({\rm i}\frac{s_l}{c_2}|p_B(\varsigma)-p_D(\eta)|)}{|p_B(\varsigma)-p_D(\eta)|}\varphi_{2,l}(\eta)
		\bigr\},
	\end{aligned}
\end{equation*}
\begin{equation*}
	\begin{aligned}
		&L_{3,1}^{(l),m}(\varsigma,\eta,c_1,c_2,\varphi_{1,l},\varphi_{2,l})\\
		=&\bigl\{\big( -\frac{{\rm i}s_l^2}{4c_1^2}\frac{ H_0^{(1)}({\rm i}\frac{s_l}{c_1}|p_B(\varsigma)-p_D(\eta)|)}{|p_B(\varsigma)-p_D(\eta)|^2}+\frac{s_l}{2c_1}\frac{H_1^{(1)}({\rm i}\frac{s_l}{c_1}|p_B(\varsigma)-p_D(\eta)|)}{|p_B(\varsigma)-p_D(\eta)|^3}\big)\\
		&\times\big(b_1+R\cos\varsigma-p_1-r(\eta)\cos\eta\big)\big[(b_1-p_1)\cos\eta+(b_2-p_2)\sin\eta+R\cos(\varsigma-\eta)-r(\eta)\big]\cos (m\eta)\varphi_{1,l}(\eta)\\
		&-\frac{s_l}{4c_1}\frac{H_1^{(1)}({\rm i}\frac{s_l}{c_1}|p_B(\varsigma)-p_D(\eta)|)}{|p_B(\varsigma)-p_D(\eta)|}\cos\eta\cos(m\eta)\varphi_{1,l}(\eta)\\
		&+\big( -\frac{{\rm i}s_l^2}{4c_2^2}\frac{ H_0^{(1)}({\rm i}\frac{s_l}{c_2}|p_B(\varsigma)-p_D(\eta)|)}{|p_B(\varsigma)-p_D(\eta)|^2}+\frac{s_l}{2c_2}\frac{H_1^{(1)}({\rm i}\frac{s_l}{c_2}|p_B(\varsigma)-p_D(\eta)|)}{|p_B(\varsigma)-p_D(\eta)|^3}\big)\\
		&\times\big(b_2+R\sin\varsigma-p_2-r(\eta)\sin\eta\big)\big[(b_1-p_1)\cos\eta+(b_2-p_2)\sin\eta+R\cos(\varsigma-\eta)-r(\eta)\big]\cos (m\eta)\varphi_{2,l}(\eta)\\
		&-\frac{s_l}{4c_2}\frac{H_1^{(1)}({\rm i}\frac{s_l}{c_2}|p_B(\varsigma)-p_D(\eta)|)}{|p_B(\varsigma)-p_D(\eta)|}\sin\eta\cos(m\eta)\varphi_{2,l}(\eta)
		\bigr\},
	\end{aligned}
\end{equation*}
\begin{equation*}
	\begin{aligned}
		&L_{4,1}^{(l),m}(\varsigma,\eta,c_1,c_2,\varphi_{1,l},\varphi_{2,l})\\
		=&\bigl\{\big( -\frac{{\rm i}s_l^2}{4c_1^2}\frac{ H_0^{(1)}({\rm i}\frac{s_l}{c_1}|p_B(\varsigma)-p_D(\eta)|)}{|p_B(\varsigma)-p_D(\eta)|^2}+\frac{s_l}{2c_1}\frac{H_1^{(1)}({\rm i}\frac{s_l}{c_1}|p_B(\varsigma)-p_D(\eta)|)}{|p_B(\varsigma)-p_D(\eta)|^3}\big)\\
		&\times\big(b_1+R\cos\varsigma-p_1-r(\eta)\cos\eta\big)\big[(b_1-p_1)\cos\eta+(b_2-p_2)\sin\eta+R\cos(\varsigma-\eta)-r(\eta)\big]\sin (m\eta)\varphi_{1,l}(\eta)\\
		&-\frac{s_l}{4c_1}\frac{H_1^{(1)}({\rm i}\frac{s_l}{c_1}|p_B(\varsigma)-p_D(\eta)|)}{|p_B(\varsigma)-p_D(\eta)|}\cos\eta\sin(m\eta)\varphi_{1,l}(\eta)\\
		&+\big( -\frac{{\rm i}s_l^2}{4c_2^2}\frac{ H_0^{(1)}({\rm i}\frac{s_l}{c_2}|p_B(\varsigma)-p_D(\eta)|)}{|p_B(\varsigma)-p_D(\eta)|^2}+\frac{s_l}{2c_2}\frac{H_1^{(1)}({\rm i}\frac{s_l}{c_2}|p_B(\varsigma)-p_D(\eta)|)}{|p_B(\varsigma)-p_D(\eta)|^3}\big)\\
		&\times\big(b_2+R\sin\varsigma-p_2-r(\eta)\sin\eta\big)\big[(b_1-p_1)\cos\eta+(b_2-p_2)\sin\eta+R\cos(\varsigma-\eta)-r(\eta)\big]\sin (m\eta)\varphi_{2,l}(\eta)\\
		&-\frac{s_l}{4c_2}\frac{H_1^{(1)}({\rm i}\frac{s_l}{c_2}|p_B(\varsigma)-p_D(\eta)|)}{|p_B(\varsigma)-p_D(\eta)|}\sin\eta\sin(m\eta)\varphi_{2,l}(\eta)
		\bigr\},
	\end{aligned}
\end{equation*}
\begin{equation*}
	\begin{aligned}
		&L_{1,2}^{(l)}(\varsigma,\eta,c_1,c_2,\varphi_{1,l},\varphi_{2,l})\\
		=&\bigl\{\big( -\frac{{\rm i}s_l^2}{4c_1^2}\frac{ H_0^{(1)}({\rm i}\frac{s_l}{c_1}|p_B(\varsigma)-p_D(\eta)|)}{|p_B(\varsigma)-p_D(\eta)|^2}+\frac{s_l}{2c_1}\frac{H_1^{(1)}({\rm i}\frac{s_l}{c_1}|p_B(\varsigma)-p_D(\eta)|)}{|p_B(\varsigma)-p_D(\eta)|^3}\big)\\
		&\times\big(b_2+R\sin\varsigma-p_2-r(\eta)\sin\eta\big)\big(b_1+R\cos\varsigma-p_1-r(\eta)\cos\eta\big)\varphi_{1,l}(\eta)\\
		&-\big(-\frac{{\rm i}s_l^2}{4c_2^2}\frac{ H_0^{(1)}({\rm i}\frac{s_l}{c_2}|p_B(\varsigma)-p_D(\eta)|)}{|p_B(\varsigma)-p_D(\eta)|^2}+\frac{s_l}{2c_2}\frac{H_1^{(1)}({\rm i}\frac{s_l}{c_2}|p_B(\varsigma)-p_D(\eta)|)}{|p_B(\varsigma)-p_D(\eta)|^3}\big) \big(b_1+R\cos\varsigma-p_1-r(\eta)\cos\eta\big)^2\varphi_{2,l}(\eta)\\
		&+\frac{s_l}{4c_2}\frac{H_1^{(1)}({\rm i}\frac{s_l}{c_2}|p_B(\varsigma)-p_D(\eta)|)}{|p_B(\varsigma)-p_D(\eta)|}\varphi_{2,l}(\eta)
		\bigr\},
	\end{aligned}
\end{equation*}
\begin{equation*}
	\begin{aligned}
		&L_{2,2}^{(l)}(\varsigma,\eta,c_1,c_2,\varphi_{1,l},\varphi_{2,l})\\
		=&\bigl\{\big( -\frac{{\rm i}s_l^2}{4c_1^2}\frac{ H_0^{(1)}({\rm i}\frac{s_l}{c_1}|p_B(\varsigma)-p_D(\eta)|)}{|p_B(\varsigma)-p_D(\eta)|^2}+\frac{s_l}{2c_1}\frac{H_1^{(1)}({\rm i}\frac{s_l}{c_1}|p_B(\varsigma)-p_D(\eta)|)}{|p_B(\varsigma)-p_D(\eta)|^3}\big)\big(b_2+R\sin\varsigma-p_2-r(\eta)\sin\eta\big)^2\varphi_{1,l}(\eta)\\
		&-\frac{s_l}{4c_1}\frac{H_1^{(1)}({\rm i}\frac{s_l}{c_1}|p_B(\varsigma)-p_D(\eta)|)}{|p_B(\varsigma)-p_D(\eta)|}\varphi_{1,l}(\eta)\\
		&-\big(-\frac{{\rm i}s_l^2}{4c_2^2}\frac{ H_0^{(1)}({\rm i}\frac{s_l}{c_2}|p_B(\varsigma)-p_D(\eta)|)}{|p_B(\varsigma)-p_D(\eta)|^2}+\frac{s_l}{2c_2}\frac{H_1^{(1)}({\rm i}\frac{s_l}{c_2}|p_B(\varsigma)-p_D(\eta)|)}{|p_B(\varsigma)-p_D(\eta)|^3}\big)\\
		&\times \big(b_2+R\sin\varsigma-p_2-r(\eta)\sin\eta\big)\big(b_1+R\cos\varsigma-p_1-r(\eta)\cos\eta\big)\varphi_{2,l}(\eta)
		\bigr\},
	\end{aligned}
\end{equation*}
\begin{equation*}
	\begin{aligned}
		&L_{3,2}^{(l),m}(\varsigma,\eta,c_1,c_2,\varphi_{1,l},\varphi_{2,l})\\
		=&\bigl\{\big( -\frac{{\rm i}s_l^2}{4c_1^2}\frac{ H_0^{(1)}({\rm i}\frac{s_l}{c_1}|p_B(\varsigma)-p_D(\eta)|)}{|p_B(\varsigma)-p_D(\eta)|^2}+\frac{s_l}{2c_1}\frac{H_1^{(1)}({\rm i}\frac{s_l}{c_1}|p_B(\varsigma)-p_D(\eta)|)}{|p_B(\varsigma)-p_D(\eta)|^3}\big)\\
		&\times\big(b_2+R\sin\varsigma-p_2-r(\eta)\sin\eta\big)\big[(b_1-p_1)\cos\eta+(b_2-p_2)\sin\eta+R\cos(\varsigma-\eta)-r(\eta)\big]\cos (m\eta)\varphi_{1,l}(\eta)\\
		&-\frac{s_l}{4c_1}\frac{H_1^{(1)}({\rm i}\frac{s_l}{c_1}|p_B(\varsigma)-p_D(\eta)|)}{|p_B(\varsigma)-p_D(\eta)|}\sin\eta\cos(m\eta)\varphi_{1,l}(\eta)\\
		&-\big( -\frac{{\rm i}s_l^2}{4c_2^2}\frac{ H_0^{(1)}({\rm i}\frac{s_l}{c_2}|p_B(\varsigma)-p_D(\eta)|)}{|p_B(\varsigma)-p_D(\eta)|^2}+\frac{s_l}{2c_2}\frac{H_1^{(1)}({\rm i}\frac{s_l}{c_2}|p_B(\varsigma)-p_D(\eta)|)}{|p_B(\varsigma)-p_D(\eta)|^3}\big)\\
		&\times\big(b_1+R\cos\varsigma-p_1-r(\eta)\cos\eta\big)\big[(b_1-p_1)\cos\eta+(b_2-p_2)\sin\eta+R\cos(\varsigma-\eta)-r(\eta)\big]\cos (m\eta)\varphi_{2,l}(\eta)\\
		&+\frac{s_l}{4c_2}\frac{H_1^{(1)}({\rm i}\frac{s_l}{c_2}|p_B(\varsigma)-p_D(\eta)|)}{|p_B(\varsigma)-p_D(\eta)|}\cos\eta\cos(m\eta)\varphi_{2,l}(\eta)
		\bigr\},
	\end{aligned}
\end{equation*}
\begin{equation*}
	\begin{aligned}
		&L_{4,2}^{(l),m}(\varsigma,\eta,c_1,c_2,\varphi_{1,l},\varphi_{2,l})\\
		=&\bigl\{\big( -\frac{{\rm i}s_l^2}{4c_1^2}\frac{ H_0^{(1)}({\rm i}\frac{s_l}{c_1}|p_B(\varsigma)-p_D(\eta)|)}{|p_B(\varsigma)-p_D(\eta)|^2}+\frac{s_l}{2c_1}\frac{H_1^{(1)}({\rm i}\frac{s_l}{c_1}|p_B(\varsigma)-p_D(\eta)|)}{|p_B(\varsigma)-p_D(\eta)|^3}\big)\\
		&\times\big(b_2+R\sin\varsigma-p_2-r(\eta)\sin\eta\big)\big[(b_1-p_1)\cos\eta+(b_2-p_2)\sin\eta+R\cos(\varsigma-\eta)-r(\eta)\big]\sin (m\eta)\varphi_{1,l}(\eta)\\
		&-\frac{s_l}{4c_1}\frac{H_1^{(1)}({\rm i}\frac{s_l}{c_1}|p_B(\varsigma)-p_D(\eta)|)}{|p_B(\varsigma)-p_D(\eta)|}\sin\eta\sin(m\eta)\varphi_{1,l}(\eta)\\
		&-\big( -\frac{{\rm i}s_l^2}{4c_2^2}\frac{ H_0^{(1)}({\rm i}\frac{s_l}{c_2}|p_B(\varsigma)-p_D(\eta)|)}{|p_B(\varsigma)-p_D(\eta)|^2}+\frac{s_l}{2c_2}\frac{H_1^{(1)}({\rm i}\frac{s_l}{c_2}|p_B(\varsigma)-p_D(\eta)|)}{|p_B(\varsigma)-p_D(\eta)|^3}\big)\\
		&\times\big(b_1+R\cos\varsigma-p_1-r(\eta)\cos\eta\big)\big[(b_1-p_1)\cos\eta+(b_2-p_2)\sin\eta+R\cos(\varsigma-\eta)-r(\eta)\big]\sin (m\eta)\varphi_{2,l}(\eta)\\
		&+\frac{s_l}{4c_2}\frac{H_1^{(1)}({\rm i}\frac{s_l}{c_2}|p_B(\varsigma)-p_D(\eta)|)}{|p_B(\varsigma)-p_D(\eta)|}\cos\eta\sin(m\eta)\varphi_{2,l}(\eta)
		\bigr\}.
	\end{aligned}
\end{equation*}

\end{document}